\documentclass[final,sort&compress,3p,times,12pt]{elsarticle}
\usepackage{amsmath}
\usepackage{amsfonts}
\usepackage{}
\usepackage{natbib}
\usepackage{amssymb}
\usepackage{amsthm}
\usepackage{graphicx}
\usepackage{subfigure}
\usepackage{algorithm}
\usepackage{algorithmic}
\usepackage{color}
\usepackage{array}
\usepackage{mathrsfs}
\usepackage{multirow}
\usepackage{graphicx}
\usepackage{subfigure}

\DeclareMathOperator{\spn}{span}
\allowdisplaybreaks[4]

\newtheorem{Theo}{Theorem}

\newtheorem{Rem}{Remark}
\newtheorem{Lem}{Lemma}

\newtheorem{Ex}{Example}
\numberwithin{equation}{section}
 \numberwithin{Lem}{section}
 \numberwithin{Defi}{section}
 \numberwithin{Theo}{section}
 \numberwithin{Rem}{section}
  \numberwithin{Coro}{section}
  \numberwithin{Fig}{section}

\def\a{\alpha}

\begin{document}

\begin{frontmatter}


\title{Unconditional energy dissipation and error  estimates of the SAV Fourier spectral method for nonlinear fractional generalized wave equation$^\star$ \tnotetext[1]{This work was supported in part by NSF of China (12001499, 11771163,11801527,12011530058), China Postdoctoral Science Foundation (2019M662506,2018M632791)
.}}
\author{Nan Wang$^{a}$,~~Meng Li$^{a,*}$,~~Chengming Huang$^{b,c}$}
\cortext[cor1]{Corresponding author (Meng Li).\\
\emph{Email addresses}: \texttt{nan\_wang1991@163.com} (N. Wang),
\texttt{limeng@zzu.edu.cn} (M. Li$^*$),\\ \texttt{chengming\_huang@hotmail.com} (C. Huang).}
 \address{$^{a}$School of Mathematics and Statistics, Zhengzhou University, Zhengzhou 450001, China}
  \address{$^{b}$School of Mathematics and Statistics, Huazhong University of Science and Technology, Wuhan 430074, China\\}
  \address{$^{c}$Hubei Key Laboratory of Engineering Modeling and Scientific Computing, Huazhong University of Science and Technology, Wuhan 430074, China\\}
\date{}
\begin{abstract}
\par
In this paper, we  consider a second-order  scalar auxiliary variable (SAV) Fourier spectral method to solve the nonlinear fractional generalized wave  equation. 
 Unconditional energy conservation or dissipation properties of the fully discrete  scheme are first established. Next, we utilize the  temporal-spatial error splitting argument to obtain  unconditional optimal error estimate of the fully discrete scheme, which overcomes time-step restrictions  caused by strongly nonlinear system, or the restrictions that the nonlinear term  needs to satisfy the assumption of global Lipschitz condition in all previous works for fractional undamped or damped wave equations. 
Finally, some numerical experiments are presented to confirm our theoretical analysis. 
\end{abstract}
\end{frontmatter}

\section{Introduction}
In this paper,
we consider the following two-dimensional nonlinear space-fractional generalized wave equation (FGWE)
 \begin{eqnarray}
    && u_{tt}+\kappa (-\Delta)^{\frac{\a}{2}}u+\gamma_{1}(-\Delta)^{\frac{\a}{2}}u_{t}+\gamma_{2}u_{t}+F'(u)=0,~~\mathbf{x}\in
   \Omega,~0<t\leq T,\label{sin1}\\
  && u(\mathbf{x},0)=\phi_{0}(\mathbf{x}),~u_{t}(\mathbf{x},0)=\varphi_{0}(\mathbf{x}),~ \mathbf{x}\in \Omega , \label{sin2}
\end{eqnarray}
where the parameters $\kappa>0$, 
$1<\alpha\leq2$, $\mathbf{x}=(x_{1},x_{2})$ and  $\gamma_{1}\geq 0, \gamma_{2} \geq 0$ are coefficients of  damping terms. $F(u)$ is nonlinear and nonnegative, and $F'(u)$ satisfies local Lipschitz condition.
The fractional Laplacian
$(-\Delta)^{\alpha}$ is defined by
\begin{equation}\label{fourierdef}
  (-\Delta)^{\a}u=\sum_{s,l\in Z}(s^{2}+l^{2})^{\a}\hat{u}_{sl}e^{isx_{1}+ilx_{2}}.
\end{equation}
  For $u\in L_{per}^{2}(\Omega)$, {\color{blue}$u=\sum_{s,l\in \mathbb{Z}}\hat{u}_{sl}e^{isx_{1}+ilx_{2}}$}, where  $i^{2}=-1$ and the Fourier coefficients {\color{blue}$\hat{u}_{sl}$} are given by
\begin{equation}\label{uhatdef}
\hat{u}_{sl}=(u,e^{isx_{1}+ilx_{2}})=\frac{1}{|\Omega|}\int_{\Omega}ue^{isx_{1}+ilx_{2}}d\mathbf{x}.
\end{equation}
In addition, \eqref{sin1}-\eqref{sin2} possess the following energy function
{\color{blue}
\begin{equation}
H(t)=\int_{\Omega}\frac{1}{2}|u_{t}|^{2}+\frac{\kappa}{2}|(-\Delta)^{\frac{\alpha}{4}}u|^{2}+F(u)d\mathbf{x},
\end{equation}
}
 and it holds that
\begin{equation}
\frac{dH(t)}{dt}=-\int_{\Omega}(\gamma_{1}|(-\Delta)^{\frac{\alpha}{2}}u_{t}|^{2}+\gamma_{2}|u_{t}|^{2})d\mathbf{x}, ~~t\in (0,T].
\end{equation}

  The conventional damped ($\gamma_{1}= 0, \gamma_2> 0$)~or undamped ($\gamma_{1}=\gamma_2=0$) wave equations are extensively described in anomalous diffusion, hydrology and so on \cite{MR2526990,MR2564022,MR3292452}.  
    If $F'(u)=\sin u$, \eqref{sin1} deduces to the damped sine-Gordon equation, and  when $F'(u)=u^{3}-1$, \eqref{sin1} becomes the damped Klein-Gordon equation.
   The FGWEs as the
generalization of conventional wave equations 
are widely  applied in science and engineering to well describe the  long-range interaction, such as the interaction of solitons in a collisionless plasma  and the presence of the phenomenon of nonlinear supratransmission of energy \cite{MR2569625}. A increasing number of  mathematical and numerical methods have been developed for the FGWE. Along the mathematical front, Shomberg \cite{MR3998923} proved the well-posedness of 
the FGWE and also derived the energy dissipation-preserving structure. In general, the analytical solution of the FGWE is difficult to obtain and we have to construct numerical methods.

  Along the numerical front, numerous efficient numerical methods for the fractional damped ($\gamma_{1}= 0, \gamma_2> 0$)~or undamped ($\gamma_{1}=\gamma_2=0$) wave equation were proposed in recent years.
  Liu et al. \cite{MR3281878}  studied a class of unconditionally stable difference schemes based on the $Pad\acute{e}$ approximation for the Riesz space-fractional telegraph equation. Ran and Zhang \cite{MR3466137} derived a compact difference scheme with accuracy of fourth-order in space and second-order in time for fractional damped wave equation.  Plenty of numerical works in \cite{MR3342471,MR3759440,li2019fast,wang2019conservative,MR3782561,MR3342469,MR3903668,WangN:2018:CMA,ZhangH:2018:IJCM}  have shown that energy  dissipative (or conservative) numerical methods have obvious superiority  over common numerical methods in long time simulation and thus developing energy dissipation-preserving methods for the fractional wave equation has attracted more  and more researchers' attention. For the fractional undamped wave equation, Xin et al. \cite{MR3826777} proposed {\color{blue} a} conservative difference scheme for Riesz
space-fractional sine-{G}ordon equation. Fu et al. \cite{MR4032777} derived {\color{blue}an} explicit structure-preserving algorithm by considering {H}amiltonian system for fractional wave equation. For the fractional damped wave equation, Mac\'{\i}as-D\'{\i}az et al. \cite{MR3758373,MR3758444,MR3850343,MR3759123} proposed a series of efficient  structure-preserving finite difference methods to
 study the fractional sine-Gordon equation with Riesz fractional derivative. 
 Xie et al. \cite{MR3946475} proposed {\color{blue}a} dissipation-preserving fourth-order difference method for fractional damped wave equation. 
 However, we noticed that all existing numerical methods were devoted to investigate the nonlinear fractional wave equation ($\gamma_{1}= 0, \gamma_2\geq 0$) and
there is a little attention for the fractional generalized wave equation ($\gamma_{1}> 0, \gamma_2> 0$). Recently,   Xie et al. \cite{MR3955186} developed a  dissipation-preserving fourth-order difference method for the nonlinear fractional generalized  wave equations with Riesz fractional derivative in two dimensions, and this is the first work to prove unconditional stability and convergence of the fully discrete scheme. But unfortunately, the above scheme is fully implicit, which needs to solve the nonlinear system by a iterative algorithm at each time step, and increases the computational cost.
  At the same time, {\color{blue} the} unconditional convergence analysis is under the assumption that $F'(u)$ must satisfy  global Lipschitz condition, {\color{blue}
which means that $|F'(x)-F'(y)|\leq L |x-y|,~x,~y \in R$, where $L$ is a positive constant independent of $x$ and $y$. However, this assumption limits the range of applicability. In fact, nonlinear terms in a very large of nonlinear wave equations are the local   Lipschitz continuous, i.e., for any $x,y \in \Omega$,  $|F'(x)-F'(y)|\leq L_{\Omega} |x-y|$, where $L_{\Omega}>0$ is a Lipschitz constant dependent on  $\Omega$. For example, if $F'(u)=u^{2}$, the nonlinear term satisfies the local Lipschitz condition.}  
  These two reasons motivate us to construct a linearly implicit scheme  to reduce computational cost and use some other analytical techniques to remove the restriction of the global Lipschitz condition.

   There are many  energy dissipative schemes in \cite{MR1848726,MR3691811,MR3402723,MR3267098} for classical nonlinear wave equations. In particular, Jiang et al. \cite{MR3995980} extended invariant energy quadratization (IEQ) approach (proposed in \cite{MR3691811}), and established a linearly implicit energy-preserving scheme for sine-Gordon equation.  Very recently, to overcome introducing an auxiliary function by the IEQ approach,  Shen and Xu \cite{MR3989239,MR3723659} proposed a new efficient linearly implicit scheme which is called the scalar auxiliary variable (SAV) approach  to describe energy dissipative physical systems without the Lipschitz assumption. Moreover, the SAV approach {\color{blue}results} in a linear system at each time step and it is easy to implement the scheme.   Next, Li and Shen \cite{li2019stability} gave rigorous error estimate  for  the stabilized SAV Fourier spectral for the phase field crystal equation.  The SAV approach is  being studied extensively for other partial differential equations, see Refs. \cite{MR4083210,LiShen} and references therein. Following the superiority  of SAV approach, we develop the fully discrete SAV scheme 
   for the FGWE \eqref{sin1}-\eqref{sin2}.   In addition, the convergence analysis plays an important role in numerical methods. Most previous works need the requirement of global Lipschitz condition and thus it is necessary to study the convergence analysis under weaker condition. Similar to the technique in \cite{MR3274369}, Wang et al. \cite{WangFei} recently provided a rigorous convergence analysis  for fractional damped wave equation under the condition $\tau^2N\leq c$ caused by the inverse inequality as usual, where nonlinear term is under the weaker assumption.
  The similar time-spatial stepsize restriction often appeared in general nonlinear parabolic equations, and in order to get the unconditional error estimate without the time stepsize restrictions,  the temporal-spatial error splitting argument  was presented to {\color{blue} get} unconditional optimal  error estimate for  parabolic equations \cite{MR3072763,MR3673699,MR3530964}. Inspired by the technique, Zeng  et al. \cite{ MR4045247} proposed Fourier spectral method  for nonlinear fractional reaction-diffusion equation and used the temporal-spatial error splitting argument  to overcome the time stepsize restriction in \cite{MR3274369}. However,  using the temporal-spatial error splitting argument to obtain the unconditional error estimate for the FGWE, has not been studied.
 In this work, we extend SAV  Fourier spectral method for the FGWE, which the resulting system can be used  FFT solver. Furthermore, the temporal-spatial error splitting argument is adopted to study the unconditional convergence analysis of the fully discrete scheme.

 The main contribution of this paper is  to develop the SAV Fourier-spectral method for the FGWE in two dimensions, which can maintain energy dissipation and reach high error accuracy. The unconditional energy dissipation is proved in detail.  We use temporal-spatial error splitting argument  to obtain  the unconditional optimal error estimate without the global Lipschitz assumption.  We obtain that our scheme is convergent with second order accuracy in time and spectral accuracy  in space.
 Numerical experiments are provided to confirm  the theoretical results and validate the efficiency of our algorithms.

 The rest of paper is organized as follows. In Section 2, we recall some technical lemmas and notations. In Section 3, we present the SAV spectral method for FGWE and main results.  Unconditional energy dissipation and  optimal error estimate results of the fully discrete scheme are shown.  We get time-discrete system based on SAV approach  and give error estimate as well as unconditional energy dissipation for time-discrete system in Section 4.
The boundedness
of numerical solutions in $L^{\infty}$ norm are proved unconditionally in Section 5. In section 6, the optimal convergence analysis of the SAV spectral scheme based on the  temporal
error estimate and the spatial error estimate is completed. The numerical experiments are performed to confirm the correctness of theoretical analysis in Section 7. Some conclusions are drawn in Section 8.

\section{Preliminaries}
In this section, we first recall some essential notations and lemmas.  Denote $\Omega$ be a finite domain with $\Omega=I_{x}\times I_{y}=(a, b)\times (c,d)$. Let $C^{\infty}_{per}(\Omega)$ be the set of all restrictions onto  $\Omega$ of all complex-values, $2\pi$-periodic, $C^{\infty}$ function on $\mathbb{R}^{2}$. For  a nonnegative real number $r$, let $H^{r}_{per}(\Omega)$ be the closure of $C_{per}^{\infty}(\Omega)$ with  {\color{blue} the  semi-norm $|\cdot|_{r}$ and norm $\|\cdot\|_{r}$ } defined by
 \begin{equation}\label{timedis1}
   |u|_{r}^{2}=\sum_{s,l\in Z}\hat{u}_{s,l}^{2}(s^{2}+l^{2})^{r},~~ \|u\|_{r}^{2}=\sum_{s,l\in Z}\hat{u}_{s,l}^{2}(1+s^{2}+l^{2})^{r}.
 \end{equation}
 \begin{Lem}\label{le1.1}
 (\cite{MR3670722}).
 Let $\xi,r \geq 0$, then for any $u,v \in H_{per}^{\xi+r}(\Omega)$, it hold that
 \begin{equation*}
   ((-\Delta)^{\xi+r}u,v)=((-\Delta)^{\xi}u,(-\Delta)^{r}v).
 \end{equation*}
 \end{Lem}
 \begin{Lem}\label{le1.2}
(\cite{MR3670722,MR4045247}). If $\xi<r$ and $u ,v \in H^{2\xi}_{per}(\Omega)$ for all $\xi>0$, then
 \begin{equation*}
 \|(-\Delta)^{\xi}(uv)\|^{2}\leq C_{\xi}[\|v\|_{\infty}^{2}\|(-\Delta)^{\xi}u\|^2]+\|(-\Delta)^{\xi}v\|^{2}\|u\|_{\infty}^2],
 \end{equation*}
 \end{Lem}
 where $C_{\xi}=max\{1,2^{2\xi-1}\}$.
 \begin{Lem}\label{le1.3}
 (\cite{Adams2003Sobolev,MR4045247}).
 Let $\Omega$ be a domain in $\mathbb{R}^{n}$ satisfying the cone condition. If $mp>n$, let $p\leq q \leq \infty$; if $mp<n$, let $p\leq q \leq q^{*}=np/(n-mp)$. Then there exists  a positive constant $C_{1}$ depending on $m,n, p,q$, and the dimensions of $\Omega$ such that for all $u\in W^{m,p}(\Omega)$,
 \begin{equation*}
   \|u\|_{q}\leq C_{1}\|u\|_{m,p}^{\varepsilon}\|u\|_{p}^{1-\varepsilon},
 \end{equation*}
 where $\varepsilon=(n/mp)-(n/mq)$.
 \end{Lem}
 \begin{Lem}\label{le1.4}
 (\cite{Shen2011Spectral,MR4045247}).
 For any $u\in X_{N}$, there existing a positive constant $C_{2}$ independent of $N$ and  the following inverse inequality holds,
 \begin{equation*}
   \|u\|_{\infty} \leq C_{2}N\|u\|.
 \end{equation*}
 \end{Lem}
 \begin{Lem}\label{le1.5}
 (\cite{MR3670722}).
 Suppose that $u\in H^{r}_{per}(\Omega)$.  Then the following estimate holds for all $0\leq \xi \leq r$,
 \begin{equation*}
   \|u-P_{N}u\|_{\xi}\leq C_{3}N^{\xi-r}\|u\|_{r},
 \end{equation*}
 where $C_{3}$ is a positive constant not depending on $N$.
 \end{Lem}

\section{The SAV Fourier spectral method for two-dimensional FGWE and main results}
In this section, 
a  linealy implicit  fully discrete scheme is constructed, which is based on the SAV approach in time and Fourier spectral method in space. Moreover, we give theoretical analysis of the linear system, including the unconditional energy dissipation and the corresponding error estimate.

\subsection{The stabilized SAV  approach}

The equation \eqref{sin1}-\eqref{sin2}, by introducing a scalar variable $r(t)=\sqrt{E(u)}$,  can be transformed into the following system 
\begin{eqnarray}
&&u_{t}=v, \label{sinene1}\\
  &&v_{t}+\kappa (-\Delta)^{\frac{\a}{2}}u+\gamma_{1}(-\Delta)^{\frac{\a}{2}}v+\gamma_{2}v+r(t)\frac{F'(u)}{\sqrt{E(u)}} =0 ,  \label{sinene2}\\
  &&r_{t}=\frac{1}{2\sqrt{E(u)}}\int_{\Omega}F'(u)u_{t}d\mathbf{x}. \label{sinene3}\\
&& u(\mathbf{x},0)=\phi(\mathbf{x},y),~v(\mathbf{x},0)=\varphi(\mathbf{x}),~ \mathbf{x}\in \Omega,  \label{sinene4}
\end{eqnarray}
where $E(u)=\int_{\Omega}F(u)d\mathbf{x}+C_{0}$, and  $C_{0}$ is chosen such that $E(u)>0$. In addition, we assume  $F \in C^{3}(\mathbb{R})$.
\begin{Theo}\label{theo2.1}
(Energy dissipation)
Under the periodic boundary conditions, \eqref{sinene1}-\eqref{sinene4} poessess the following energy function
\begin{equation}\label{saveng}
H(t)=\int_{\Omega}\frac{1}{2}|u_{t}|^{2}+\frac{\kappa}{2}|(-\Delta)^{\frac{\alpha}{4}}u|^{2}d\mathbf{x}+r^{2},
\end{equation}
 and it holds that
 \begin{eqnarray}
  &\frac{dH(t)}{dt}=-\int_{\Omega}\Big(\gamma_{1}|(-\Delta)^{\frac{\alpha}{2}}u_{t}|^{2}+\gamma_{2}|u_{t}|^{2}\Big)d\mathbf{x} \leq 0, \label{saveng1}
 \end{eqnarray}
 and thus
 \begin{equation}\label{saveng2}
   H(t_{2})\leq H(t_{1}),~t_{1}>t_{2}.
 \end{equation}
%
\begin{proof}
  Taking the inner product of \eqref{sinene2} with $v$ and using \eqref{sinene1}, we can directly obtain
  \begin{equation*}
    (v_{t},v)+\kappa ((-\Delta)^{\frac{\a}{2}}u,u_{t})+\gamma_{1}((-\Delta)^{\frac{\a}{2}}v,v)+\gamma_{2}(v,v)+r(t)\frac{\Big(F'(u),u_{t}\Big)}{\sqrt{E(u)}} =0.
  \end{equation*}
  Combining with \eqref{sinene1} and Leibiniz rule,  we have
  \begin{eqnarray*}
   \int_{\Omega}\frac{1}{2}|u_{t}|^{2}+\frac{\kappa}{2}|(-\Delta)^{\frac{\alpha}{4}}u|^{2}d\mathbf{x}+r(t)^{2}+\int_{\Omega}\Big(\gamma_{1}|(-\Delta)^{\frac{\alpha}{2}}u_{t}|^{2}+\gamma_{2}|u_{t}|^{2}\Big)d\mathbf{x} =0,
  \end{eqnarray*}
  and one can get \eqref{saveng} and \eqref{saveng1}-\eqref{saveng2}.
\end{proof}
\end{Theo}
\begin{Rem}
When $\gamma_{1}=\gamma_{2}=0$, the system \eqref{sin1}-\eqref{sin2} deduces to the fractional wave equation. From \eqref{saveng1}, we have $$\frac{dH(t)}{dt}=0,$$ which implies that the energy is conservative.
\end{Rem}
For a positive integer $N$, the function space is denoted by
\begin{equation*}
  X_{N}=\spn\{e^{isx_{1}+ilx_{2}}:-N/2\leq s,l \leq N/2-1\}.
\end{equation*}
Define the orthogonal projection operator $P_{N}$ as follows
\begin{eqnarray*}
  &&(u-P_{N}u,v)=0, \forall v \in X_{N}, ~u \in L^{2}_{per}(\Omega),\\
  &&(-\Delta)^{\alpha/2}P_{N}u=P_{N}(-\Delta)^{\alpha/2}u.
\end{eqnarray*}
 For the temporal discretization, we divide the interval $[0,T]$ 
  by a time step size $\tau=T/K$. Let $t_{n}=n\tau$, $u^{n}=u(\mathbf{x},t_{n})$, $0\leq n \leq K$, and denote
\begin{equation}\label{timedi2}
  \delta_{t}u^{n+\frac{1}{2}}=\frac{u^{n+1}-u^{n}}{\tau},~ \tilde{u}^{n+\frac{1}{2}}=\frac{3u^{n}-u^{n-1}}{2},~\bar{u}^{n+\frac{1}{2}}=\frac{u^{n+1}+u^{n}}{2},
\end{equation}

\subsection{The full discrete SAV Fourier spectral method}
The full discrete Fourier spectral method for \eqref{sin1}-\eqref{sin2} is given by: Find $(u_{N}^{n}, ~v_{N}^{n}) \in X_{N} \times X_{N}$, such that for all
$(\psi,~\varphi) \in X_{N} \times X_{N} $,
\begin{eqnarray}
 &&(\delta_{t}u_{N}^{n+\frac{1}{2}},\psi)=(\bar{v}_{N}^{n+\frac{1}{2}},\psi), \label{22}\\
  &&(\delta_{t}v_{N}^{n+\frac{1}{2}},\varphi)+\kappa ((-\Delta)^{\frac{\a}{2}}\bar{u}_{N}^{n+\frac{1}{2}},\varphi)+\gamma_{1}((-\Delta)^{\frac{\a}{2}}\bar{v}_{N}^{n+\frac{1}{2}},\varphi)+ \nonumber\\
  &&\gamma_{2}(\bar{v}_{N}^{n+\frac{1}{2}},\varphi)
  +\bar{R}^{n+\frac{1}{2}}\Big(\frac{F'(\tilde{u}_{N}^{n+\frac{1}{2}})}{\sqrt{E(\tilde{u}_{N}^{n+\frac{1}{2}})}},\varphi\Big) =0 ,  \label{23}\\
  &&\delta_{t}R^{n+\frac{1}{2}}=\frac{1}{2\sqrt{E(\tilde{u}_{N}^{n+\frac{1}{2}})}}\Big(F'(\tilde{u}_{N}^{n+\frac{1}{2}}),\delta_{t}u_{N}^{n+\frac{1}{2}}\Big). \label{24}\\
  &&u_{N}^{0}(\mathbf{x})=P_{N}u^{0}(\mathbf{x}), v_{N}^{0}(x)=P_{N}v^{0}(\mathbf{x}), R^{0}=r^{0},~ \mathbf{x}\in \Omega.  \label{244}
\end{eqnarray}
Since the fully discrete system \eqref{22}-\eqref{24} is not self-starting,  the following scheme is used in the first step:
\begin{eqnarray}
 &&(\frac{\tilde{u}_{N}^{\frac{1}{2}}-u_{N}^{0}}{\tau/2},\psi)=(\tilde{v}_{N}^{\frac{1}{2}},\psi), \label{25}\\
  &&\bigg(\frac{\tilde{v}_{N}^{\frac{1}{2}}-v_{N}^{0}}{\tau/2},\varphi\bigg)+\kappa ((-\Delta)^{\frac{\a}{2}}\tilde{u}_{N}^{\frac{1}{2}},\varphi)+\gamma_{1}((-\Delta)^{\frac{\a}{2}}\tilde{v}_{N}^{\frac{1}{2}},\varphi)
  +\gamma_{2}(\tilde{v}_{N}^{\frac{1}{2}},\varphi)
  +\tilde{R}^{\frac{1}{2}}\bigg(\frac{F'(u_{N}^{0})}{\sqrt{E(u_{N}^{0})}},\varphi\bigg) =0 ,  \label{26}\\
  &&\frac{\tilde{R}^{\frac{1}{2}}-R^{0}}{\tau/2}=\frac{1}{2\sqrt{E(u_{N}^{0})}}\bigg(F'(u_{N}^{0}),\frac{\tilde{u}_{N}^{\frac{1}{2}}-u_{N}^{0}}{\tau/2}\bigg). \label{27}
\end{eqnarray}
\begin{Theo}\label{theodiseng}
(Unconditional full discrete energy dissipation)
  Under the periodic boundary conditions, \eqref{22}-\eqref{24} is dissipative in the sense that
  \begin{equation}\label{diseng}
\mathbf{H}^{n}\leq \mathbf{H}^{n-1}, ~1\leq n \leq K+1,
\end{equation}
where
\begin{equation}\label{diseng1}
\mathbf{H}^{n}=\frac{1}{2}\|v_{N}^{n}\|^{2}+\frac{\kappa}{2}\|(-\Delta)^{\frac{\alpha}{4}}u_{N}^{n}\|^{2}+(R^{n})^{2}.
\end{equation}
\end{Theo}
\begin{proof}
Similar to the proof for Theorem \ref{theo2.1}, letting $\varphi=\bar{v}_{N}^{n+\frac{1}{2}}$ in \eqref{23}, we  have
\begin{eqnarray*}
(\delta_{t}v_{N}^{n+\frac{1}{2}},\bar{v}_{N}^{n+\frac{1}{2}})+\kappa ((-\Delta)^{\frac{\a}{2}}\bar{u}_{N}^{n+\frac{1}{2}},\bar{v}_{N}^{n+\frac{1}{2}})+\gamma_{1}\|(-\Delta)^{\frac{\a}{4}}\bar{v}_{N}^{n+\frac{1}{2}}\|^{2}+
\gamma_{2}\|\bar{v}_{N}^{n+\frac{1}{2}}\|^2
  +\bar{R}^{n+\frac{1}{2}}\frac{\Big(F'(\tilde{u}_{N}^{n+\frac{1}{2}}),\bar{v}_{N}^{n+\frac{1}{2}}\Big)}{\sqrt{E(\tilde{u}_{N}^{n+\frac{1}{2}})}}=0 .
\end{eqnarray*}
Taking $\psi=(-\Delta)^{\frac{\a}{2}}\bar{u}_{N}^{n+\frac{1}{2}}$ in \eqref{22}, it yields
\begin{equation*}
(\delta_{t}u_{N}^{n+\frac{1}{2}},(-\Delta)^{\frac{\a}{2}}\bar{u}_{N}^{n+\frac{1}{2}})
=(\bar{v}_{N}^{n+\frac{1}{2}},(-\Delta)^{\frac{\a}{2}}\bar{u}_{N}^{n+\frac{1}{2}}).
\end{equation*}
Using
\begin{equation}\label{normrela}
(\delta_{t}v_{N}^{n+\frac{1}{2}},\bar{v}_{N}^{n+\frac{1}{2}})=\frac{1}{2\tau}(\|v_{N}^{n+1}\|^2-\|v_{N}^{n}\|^2),~
(\delta_{t}u_{N}^{n+\frac{1}{2}},(-\Delta)^{\frac{\a}{2}}\bar{u}_{N}^{n+\frac{1}{2}})=\frac{1}{2\tau}(|u^{n+1}_{N}|^{2}-|u^{n}_{N}|^{2}),
\end{equation}
and \eqref{24}, it directly achieves
\begin{small}
\begin{equation*}
  \Big(\frac{1}{2}\|v_{N}^{n+1}\|^{2}+\frac{\kappa}{2}\|(-\Delta)^{\frac{\alpha}{4}}u_{N}^{n+1}\|^{2}+(R^{n+1})^{2}\Big)-\Big(\frac{1}{2}\|v_{N}^{n}\|^{2}
  +\frac{\kappa}{2}\|(-\Delta)^{\frac{\alpha}{4}}u_{N}^{n}\|^{2}+(R^{n})^{2}\Big)+\tau\gamma_{1}\|(-\Delta)^{\frac{\a}{4}}\bar{v}_{N}^{n+\frac{1}{2}}\|^{2}+
\tau\gamma_{2}\|\bar{v}_{N}^{n+\frac{1}{2}}\|^2=0.
\end{equation*}
 \end{small}
\end{proof}

Throughout the paper, denote
C a generic positive constant independent of K and N that may has different values in different
cases. We assume that
\begin{equation}\label{bound1}
  \|u_{0}\|_{H^{m}}+\|u\|_{H^{m}}+|r|_{L^{\infty}(0,T)} +|r_{t}|_{L^{\infty}(0,T)}+|r_{tt}|_{L^{\infty}(0,T)} +|r_{ttt}|_{L^{\infty}(0,T)}+\|u_{t}\|_{H^{m}}+\|u\|_{\a}+\|u_{t}\|_{\a}+\|u_{tt}\|_{\a}+\|u_{ttt}\|_{\a/2}\leq L,
\end{equation}
 where $L$ is a positive constant independent of $N$ and $\tau$.

With the assumption, we present unconditional optimal  error estimates for the FGWE in the following theorem. For  simplicity
, we assume $u^{n}:=u(x,t_{n}) $.
{\color{blue}{  \begin{Theo}\label{theo3.3}
Let $u$ and $\{u_{N}^{n}\}_{n=0}^{K}$ be the solutions of \eqref{sin1}-\eqref{sin2} and \eqref{22}-\eqref{27},respectively. Assume $u$ satisfies \eqref{bound1}. 
Then  it holds that \begin{equation}\label{55}
  \|u^{n}-u_{N}^{n}\|^2+\|v^{n}-v_{N}^{n}\|^2+|r^{n}-R^{n}|^2 \leq C(\tau^4+N^{-2m}), ~ \kappa|u^{n}-u_{N}^{n}|^2_{\frac{\a}{2}}\leq C(\tau^4+N^{\a-2m}),
\end{equation}
where $C$ are  positive constants independent of $N$ and $\tau$.
 \end{Theo}
 }}
Taking technique similar as \cite{MR4045247}, we prove the Theorem \ref{theo3.3} in the next  two sections. We splitting error into two parts, i.e., temporal error and spatial error.
\section{Error analysis for the time-discrete system}
 In this section, we present and analyze the time-discrete  system for the nonlinear FGWE.
We get the following time-discrete scheme based on SAV approach, which define $U^{n}$, $V^{n}$ and $R^{n}$ such that
\begin{eqnarray}
 &&\delta_{t}U^{n+\frac{1}{2}}=\bar{V}^{n+\frac{1}{2}}, \label{sinets1}\\
  &&\delta_{t}V^{n+\frac{1}{2}}+\kappa (-\Delta)^{\frac{\a}{2}}\bar{U}^{n+\frac{1}{2}}+\gamma_{1}(-\Delta)^{\frac{\a}{2}}\bar{V}^{n+\frac{1}{2}}+\gamma_{2}\bar{V}^{n+\frac{1}{2}}
  +\bar{R}^{n+\frac{1}{2}}\frac{F'(\tilde{U}^{n+\frac{1}{2}})}{\sqrt{E(\tilde{U}^{n+\frac{1}{2}})}} =0 ,  \label{sinets2}\\
  &&\delta_{t}R^{n+\frac{1}{2}}=\frac{1}{2\sqrt{E(\tilde{U}^{n+\frac{1}{2}})}}\Big(F'(\tilde{U}^{n+\frac{1}{2}}),\delta_{t}U^{n+\frac{1}{2}}\Big), \label{sinets3}\\
  &&U^{0}(\mathbf{x})=u^{0}(\mathbf{x}), V^{0}(x)=v^{0}(\mathbf{x}), R^{0}=r^{0},~ \mathbf{x}\in \Omega.
\end{eqnarray}
While the first step is obtained by
\begin{eqnarray}
 &&\frac{\tilde{U}^{\frac{1}{2}}-U^{0}}{\tau/2}=\tilde{V}^{\frac{1}{2}}, \label{sinets11}\\
  &&\frac{\tilde{V}^{\frac{1}{2}}-V^{0}}{\tau/2}+\kappa (-\Delta)^{\frac{\a}{2}}\tilde{U}^{\frac{1}{2}}+\gamma_{1}(-\Delta)^{\frac{\a}{2}}\tilde{V}^{\frac{1}{2}}+\gamma_{2}\tilde{V}^{\frac{1}{2}}
  +\tilde{R}^{\frac{1}{2}}\frac{F'(U^{0})}{\sqrt{E(U^{0})}} =0 ,  \label{sinets22}\\
  &&\frac{\tilde{R}^{\frac{1}{2}}-R^{0}}{\tau/2}=\frac{1}{2\sqrt{E(U^{0})}}\Big(F'(U^{0}),\frac{\tilde{U}^{\frac{1}{2}}-U^{0}}{\tau/2}\Big). \label{sinets33}
\end{eqnarray}
{\color{blue}{
By using the \eqref{timedi2} and denote $\tilde{b}^{n}:= \frac{F'(\tilde{U}^{n+\frac{1}{2}})}{\sqrt{E(\tilde{U}^{n+\frac{1}{2}})}}$, we can get a linear equation from the SAV  scheme \eqref{sinets1}-\eqref{sinets3},
\begin{equation}
  AU^{n+1}+\frac{\tau^{2}}{4}\big(\tilde{b}^{n},U^{n+1}\big)\tilde{b}^{n}=g^{n},\label{algr1}
\end{equation}
where $A=(2+\tau \gamma_{2})I+(\frac{\tau^{2}}{2}\kappa +\tau \gamma_{1})(-\Delta)^{\a/2}$, $g^{n}=\big(A-\tau^{2}\kappa(-\Delta)^{\a/2}\big)U^{n}+2\tau V^{n}+
\frac{\tau^{2}}{4}\big(\tilde{b}^{n},U^{n}\big)\tilde{b}^{n}-\tau^{2}R^{n}\tilde{b}^{n}$.
We can solve above linear equation by first computing $A^{-1}\tilde{b}^{n}$ and $A^{-1}g^{n}$. Thus, it leads to
\begin{equation}
   U^{n+1}+\frac{\tau^{2}}{4}\big(\tilde{b}^{n},U^{n+1}\big)A^{-1}\tilde{b}^{n}=A^{-1}g^{n}.\label{algr2}
\end{equation}
Taking the inner product with $\tilde{b}^{n}$ on the both sides of the above resulting equation to solve $\big(\tilde{b}^{n},U^{n+1}\big)$, then, we can get $U^{n+1}$. Similar to the same solving procedure, we can get $\tilde{U}^{\frac{1}{2}}$ from the Eqs. \eqref{sinets11}- \eqref{sinets33} (see \cite{MR3723659} for more details).
In summary, we solve the schemes \eqref{sinets1}-\eqref{sinets3} by the following main procedure at each time step and the solving procedure of schemes \eqref{sinets11}-\eqref{sinets33} is similar.
\begin{algorithm*}
{\color{blue}
\begin{algorithmic}
  \STATE 1. Assume $U^{n}$, $R^{n}$, $U^{n-1}$, $V^{n}$ and $R^{n-1}$ are known;
  \STATE 2. Compute $A^{-1}\tilde{b}^{n}$ and $A^{-1}g^{n}$;
  \STATE 3. Compute $\big(\tilde{b}^{n},U^{n+1}\big)$ by solving the resulting equation, which is derived from taking the inner product with $\tilde{b}^{n}$ on the both sides of \eqref{algr2};
 \STATE 4. Solve  \eqref{algr1} to obtain $U^{n+1}$;
 \STATE 5. Compute $V^{n+1}$ by solving \eqref{sinets1};
 \STATE 6. Solve \eqref{sinets3} to get $R^{n+1}$.
  \end{algorithmic}}
\end{algorithm*}
}}
\begin{Theo}\label{theosemeng}
(Unconditional energy dissipation)
  Under the periodic boundary conditions, \eqref{sinets1}-\eqref{sinets3} is dissipative in the sense that
  \begin{equation}\label{semeng}
\mathbb{H}^{n+1}\leq \mathbb{H}^{n}, ~0\leq n \leq K,
\end{equation}
where
\begin{equation}\label{semeng1}
\mathbb{H}^{n}=\int_{\Omega}\frac{1}{2}|V^{n}|^{2}+\frac{\kappa}{2}|(-\Delta)^{\frac{\alpha}{4}}U^{n}|^{2}d\mathbf{x}+(R^{n})^{2}.
\end{equation}
\begin{proof}\label{pro1}
Taking the inner product of \eqref{sinets2} with $\bar{V}^{n+\frac{1}{2}}$, combining with \eqref{sinets1}, we get
\begin{equation*}
  (\delta_{t}V^{n+\frac{1}{2}},\bar{V}^{n+\frac{1}{2}})+\kappa ((-\Delta)^{\frac{\a}{2}}\bar{U}^{n+\frac{1}{2}},\delta_{t}U^{n+\frac{1}{2}})+\gamma_{1}\|(-\Delta)^{\frac{\a}{4}}\bar{V}^{n+\frac{1}{2}}\|^2+\gamma_{2}\|\bar{V}^{n+\frac{1}{2}}\|^2
  +\bar{R}^{n+\frac{1}{2}}\frac{\Big(F'(\tilde{U}^{n+\frac{1}{2}}),\bar{V}^{n+\frac{1}{2}}\Big)}{\sqrt{E(\tilde{U}^{n+\frac{1}{2}})}} =0.
\end{equation*}
By using \eqref{sinets3} and
\begin{eqnarray*}
 (\delta_{t}V^{n+\frac{1}{2}},\bar{V}^{n+\frac{1}{2}})=\frac{1}{2\tau}\Big(\|V^{n+1}\|^2-\|V^{n}\|^2\Big),~((-\Delta)^{\frac{\a}{2}}\bar{U}^{n+\frac{1}{2}},\delta_{t}U^{n+\frac{1}{2}})=\frac{1}{2\tau}\Big(\|(-\Delta)^{\frac{\alpha}{4}}U^{n+1}\|^{2}-\|(-\Delta)^{\frac{\alpha}{4}}U^{n}\|^{2}\Big),
\end{eqnarray*}
we easily obtain \eqref{semeng}.
\end{proof}
\end{Theo}

\subsection{Error estimate for time discrete scheme}

%
Denote $e_{u}^{n}=u^{n}-U^{n}$, $e_{v}^{n}=v^{n}-V^{n}$ and  $e_{r}^{n}=r^{n}-R^{n}$.
\begin{Theo}\label{theo31}
Suppose that \eqref{sin1}-\eqref{sin2} has the unique solution $(u,v,r) \in H^{\a}_{per}(\Omega)\times H^{\a}_{per}(\Omega) \times C^{3}(0,T)$. 
The assumption \eqref{bound1} and $U^{0} \in H^{\a}$ hold. Then for $0\leq n\leq K$, 
there exists a positive constant $\tau_{1}^{*}$,  such that  \eqref{sinets1}-\eqref{sinets33}  admits  unique solution {\color{blue}$(U^{n},V^{n})\in  H^{\a}_{per}(\Omega)\times H^{\a}_{per}(\Omega)$ } and for $\tau < \tau^{*}_{1}$,
\begin{align}
& \|e_{u}^{n}\|^2+\kappa |e_{u}^{n}|_{\frac{\a}{2}}^2+\kappa |e_{u}^{n}|_{\a}^2+\|e_{v}^{n}\|^2+|e_{v}^{n}|_{\frac{\a}{2}}^2
 +(e_{r}^{n})^{2}
\leq C\tau^3, ~ n=0,\ldots,K, \label{errtheo1}\\
&  \|U^{n}\|_{\infty}\leq M,~  n=0,\ldots,K, \label{errtheo2}
\end{align}
where $M>0$ is a bounded constant independent of $K$.
\end{Theo}
\begin{proof}
The \eqref{sinets1}-\eqref{sinets33} can be rewritten as the following elliptic systems
$$\left\{
     \begin{aligned}
&\frac{2\delta_{t}U^{n+\frac{1}{2}}-2V^{n}}{\tau}+\kappa (-\Delta)^{\frac{\a}{2}}\bar{U}^{n+\frac{1}{2}}+\gamma_{1}(-\Delta)^{\frac{\a}{2}}\delta_{t}U^{n+\frac{1}{2}}+\gamma_{2}\delta_{t}U^{n+\frac{1}{2}}
  +\bar{R}^{n+\frac{1}{2}}\frac{F'(\tilde{U}^{n+\frac{1}{2}})}{\sqrt{E(\tilde{U}^{n+\frac{1}{2}})}} =0, \label{semts1}\\
&\delta_{t}R^{n+\frac{1}{2}}=\frac{1}{2\sqrt{E(\tilde{U}^{n+\frac{1}{2}})}}\Big(F'(\tilde{U}^{n+\frac{1}{2}}),\delta_{t}U^{n+\frac{1}{2}}\Big), ~\forall~  n=1,2,\cdots,K, \label{semts3}\\
&\frac{\tilde{U}^{\frac{1}{2}}-U^{0}-\frac{\tau}{2}V^{0}}{\frac{\tau^2}{4}}+\kappa (-\Delta)^{\frac{\a}{2}}\tilde{U}^{\frac{1}{2}}+\gamma_{1}(-\Delta)^{\frac{\a}{2}}\frac{\tilde{U}^{\frac{1}{2}}-U^{0}}{\tau/2}+\gamma_{2}\frac{\tilde{U}^{\frac{1}{2}}-U^{0}}{\tau/2}
  +\tilde{R}^{\frac{1}{2}}\frac{F'(U^{0})}{\sqrt{E(U^{0})}} =0 ,  \label{semts4}\\
&\frac{\tilde{R}^{\frac{1}{2}}-R^{0}}{\tau/2}=\frac{1}{2\sqrt{E(U^{0})}}\Big(F'(U^{0}),\frac{\tilde{U}^{\frac{1}{2}}-U^{0}}{\tau/2}\Big). \label{semts5}
\end{aligned}
\right.$$

The existence and uniqueness of the solution to the linear elliptic equations are straightforward.
The \eqref{sinene1}-\eqref{sinene3} for time discretization at $t_{n+\frac{1}{2}}$ follow that
\begin{eqnarray}
 &&\delta_{t}u^{n+\frac{1}{2}}=\bar{v}^{n+\frac{1}{2}}+Q_{1}^{n+\frac{1}{2}}, \label{sinetay1}\\
  &&\delta_{t}v^{n+\frac{1}{2}}+\kappa (-\Delta)^{\frac{\a}{2}}\bar{u}^{n+\frac{1}{2}}+\gamma_{1}(-\Delta)^{\frac{\a}{2}}\bar{v}^{n+\frac{1}{2}}+\gamma_{2}\bar{v}^{n+\frac{1}{2}}
  +\bar{r}^{n+\frac{1}{2}}\frac{F'(\tilde{u}^{n+\frac{1}{2}})}{\sqrt{E(\tilde{u}^{n+\frac{1}{2}})}} =Q_{2}^{n+\frac{1}{2}} ,  \label{sinetay2}\\
  &&\delta_{t}r^{n+\frac{1}{2}}=\frac{1}{2\sqrt{E(\tilde{u}^{n+\frac{1}{2}})}}(F'(\tilde{u}^{n+\frac{1}{2}}),\delta_{t}u^{n+\frac{1}{2}}) +Q_{3}^{n+\frac{1}{2}},\label{sinetay3}
\end{eqnarray}
where 
\begin{align*}
&Q_{1}^{n+\frac{1}{2}}=(\delta_{t}u^{n+\frac{1}{2}}-u_{t}^{n+\frac{1}{2}})-(\bar{v}^{n+\frac{1}{2}}-v^{n+\frac{1}{2}}), \\
&Q_{2}^{n+\frac{1}{2}}=(\delta_{t}v^{n+\frac{1}{2}}-v_{t}^{n+\frac{1}{2}})
+\kappa\Big((-\Delta)^{\frac{\a}{2}}\bar{u}^{n+\frac{1}{2}}-(-\Delta)^{\frac{\a}{2}}u^{n+\frac{1}{2}}\Big)
+\gamma_{1}\Big((-\Delta)^{\frac{\a}{2}}\bar{v}^{n+\frac{1}{2}}-(-\Delta)^{\frac{\a}{2}}v^{n+\frac{1}{2}}]\Big)\\
&~~~~~~~~+\gamma_{2}(\bar{v}^{n+\frac{1}{2}}-v^{n+\frac{1}{2}})
+\bigg(\bar{r}^{n+\frac{1}{2}}\frac{F'(\tilde{u}^{n+\frac{1}{2}})}{\sqrt{E(\tilde{u}^{n+\frac{1}{2}})}}
-r^{n+\frac{1}{2}}\frac{F'(u^{n+\frac{1}{2}})}{\sqrt{E(u^{n+\frac{1}{2}})}}\bigg),\\
&Q_{3}^{n+\frac{1}{2}}=(\delta_{t}r^{n+\frac{1}{2}}-r_{t}^{n+\frac{1}{2}})-
\bigg(\frac{1}{2\sqrt{E(\tilde{u}^{n+\frac{1}{2}})}}(F'(\tilde{u}^{n+\frac{1}{2}}),\delta_{t}u^{n+\frac{1}{2}})
-\frac{1}{2\sqrt{E(u^{n+\frac{1}{2}})}}\int_{\Omega}F'(u^{n+\frac{1}{2}})u_{t}^{n+\frac{1}{2}}d\mathbf{x}\bigg).
\end{align*}
By \eqref{sinene1}-\eqref{sinene3} for time discretization at $t_{\frac{1}{2}}$, {\color{blue}{it follows that}}
\begin{eqnarray}
 &&\frac{\tilde{u}^{\frac{1}{2}}-u^{0}}{\tau/2}=\tilde{v}^{\frac{1}{2}}+\tilde{Q}_{1}^{\frac{1}{2}}, \label{sinetay11}\\
  &&\frac{\tilde{v}^{\frac{1}{2}}-v^{0}}{\tau/2}+\kappa (-\Delta)^{\frac{\a}{2}}\tilde{u}^{\frac{1}{2}}+\gamma_{1}(-\Delta)^{\frac{\a}{2}}\tilde{v}^{\frac{1}{2}}+\gamma_{2}\tilde{v}^{\frac{1}{2}}
  +\tilde{r}^{\frac{1}{2}}\frac{F'(u^{0})}{\sqrt{E(u^{0})}} =\tilde{Q}_{2}^{\frac{1}{2}},  \label{sinetay22}\\
  &&\frac{\tilde{r}^{\frac{1}{2}}-r^{0}}{\tau/2}=\frac{1}{2\sqrt{E(u^{0})}}\Big(F'(u^{0}),\frac{\tilde{u}^{\frac{1}{2}}-u^{0}}{\tau/2}\Big)+\tilde{Q}_{3}^{\frac{1}{2}}, \label{sinetay33}
\end{eqnarray}
where
\begin{align*}
&\tilde{Q}_{1}^{\frac{1}{2}}=\Big(\frac{\tilde{u}^{\frac{1}{2}}-u^0}{\tau/2}-u_{t}^{\frac{1}{2}}\Big)-(\tilde{v}^{\frac{1}{2}}-v^{\frac{1}{2}}),\\
&\tilde{Q}_{2}^{\frac{1}{2}}=\Big(\frac{\tilde{v}^{\frac{1}{2}}-v^0}{\tau/2}-v_{t}^{\frac{1}{2}}\Big)
+\kappa\Big((-\Delta)^{\frac{\a}{2}}\tilde{u}^{\frac{1}{2}}-(-\Delta)^{\frac{\a}{2}}u^{\frac{1}{2}}\Big)
+\gamma_{1}\Big((-\Delta)^{\frac{\a}{2}}\tilde{v}^{\frac{1}{2}}-(-\Delta)^{\frac{\a}{2}}v^{\frac{1}{2}}\Big)\\
&~~~~~~+\gamma_{2}(\tilde{v}^{\frac{1}{2}}-v^{\frac{1}{2}})
+\bigg(\tilde{r}^{\frac{1}{2}}\frac{F'(\tilde{u}^{0})}{\sqrt{E(\tilde{u}^{0)}}}
-r^{\frac{1}{2}}\frac{F'(u^{\frac{1}{2}})}{\sqrt{E(u^{\frac{1}{2}})}}\bigg),\\
&\tilde{Q}_{3}^{\frac{1}{2}}=(\frac{\tilde{r}^{\frac{1}{2}}-r^{0}}{\tau/2}-r_{t}^{\frac{1}{2}})-
\bigg(\frac{1}{2\sqrt{E(u^{0})}}(F'(u^{0}),\frac{\tilde{u}^{\frac{1}{2}}-u^{0}}{\tau/2})
-\frac{1}{2\sqrt{E(u^{\frac{1}{2}})}}\int_{\Omega}F'(u^{n+\frac{1}{2}})u_{t}^{n+\frac{1}{2}}\bigg).
\end{align*}
 Using Taylor formula, we get
 \begin{eqnarray}
\tau \sum_{n=0}^{K}\bigg(\|Q_{1}^{n+\frac{1}{2}}\|+\|Q_{2}^{n+\frac{1}{2}}\|+\|Q_{3}^{n+\frac{1}{2}}\|\bigg)
+\tau \bigg(\|\tilde{Q}_{1}^{\frac{1}{2}}\|+\|\tilde{Q}_{2}^{\frac{1}{2}}\|+\|\tilde{Q}_{3}^{\frac{1}{2}}\|\bigg)\leq C\tau^2.
 \end{eqnarray}

  We will give the proof sketch for \eqref{errtheo1}-\eqref{errtheo2}.

1. The first step: error estimate for $\tilde{e}_{u}^{\frac{1}{2}}=\tilde{u}^{\frac{1}{2}}-\tilde{U}^{\frac{1}{2}}$, $\tilde{e}_{v}^{\frac{1}{2}}=\tilde{v}^{\frac{1}{2}}-\tilde{V}^{\frac{1}{2}}$, $\tilde{e}_{r}^{\frac{1}{2}}=\tilde{r}^{\frac{1}{2}}-\tilde{R}^{\frac{1}{2}}$.

Combining \eqref{sinetay11}-\eqref{sinetay33} and \eqref{sinets11}-\eqref{sinets33}, we arrive at
 \begin{eqnarray}
 &&\frac{\tilde{e_{u}}^{\frac{1}{2}}}{\tau/2}=\tilde{e}_{v}^{\frac{1}{2}}+\tilde{Q}_{1}^{\frac{1}{2}}, \label{tmdstay1}\\
  &&\frac{\tilde{e}_{v}^{\frac{1}{2}}}{\tau/2}+\kappa (-\Delta)^{\frac{\a}{2}}\tilde{e}_{u}^{\frac{1}{2}}+\gamma_{1}(-\Delta)^{\frac{\a}{2}}\tilde{e}_{v}^{\frac{1}{2}}
  +\gamma_{2}\tilde{e}_{v}^{\frac{1}{2}}
  + \tilde{G}^{\frac{1}{2}}=\tilde{Q}_{2}^{\frac{1}{2}},  \label{tmdstay2}\\
 && \frac{\tilde{e}_{r}^{\frac{1}{2}}}{\tau/2}=\tilde{F}^{\frac{1}{2}}+\tilde{Q}_{3}^{\frac{1}{2}}, \label{tmdstay3}
\end{eqnarray}
where $$\tilde{G}^{\frac{1}{2}}=\tilde{r}^{\frac{1}{2}}\frac{F'(u^{0})}{\sqrt{E(u^{0})}}-\tilde{R}^{\frac{1}{2}}\frac{F'(U^{0})}{\sqrt{E(U^{0})}},$$
 and $$\tilde{F}^{\frac{1}{2}}=\frac{1}{2\sqrt{E(u^{0})}}\Big(F'(u^{0}),\frac{\tilde{u}^{\frac{1}{2}}-u^{0}}{\tau/2}\Big)
  -\frac{1}{2\sqrt{E(U^{0})}}\Big(F'(U^{0}),\frac{\tilde{U}^{\frac{1}{2}}-U^{0}}{\tau/2}\Big).$$
   Taking the inner product of  $\tilde{e}_{u}^{\frac{1}{2}}$, $\tilde{e}_{v}^{\frac{1}{2}}$ in \eqref{tmdstay1} and \eqref{tmdstay2} respectively, as well as multiplying $\tilde{e}_{r}^{\frac{1}{2}}$  on the both sides of  \eqref{tmdstay3}.
\begin{align}
&\|\tilde{e}_{u}^{\frac{1}{2}}\|^2=\frac{\tau}{2}(\tilde{e}_{u}^{\frac{1}{2}},\tilde{e}_{v}^{\frac{1}{2}})+\frac{\tau}{2} (\tilde{Q}_{1}^{\frac{1}{2}},\tilde{e}_{u}^{\frac{1}{2}}),  \label{tmerr1} \\
& \|\tilde{e}_{v}^{\frac{1}{2}}\|^2+\frac{\tau}{2}\kappa ((-\Delta)^{\frac{\a}{2}}\tilde{e}_{u}^{\frac{1}{2}},\tilde{e}_{v}^{\frac{1}{2}})
+\frac{\tau}{2}\gamma_{1}|\tilde{e}_{v}^{\frac{1}{2}}|_{\frac{\a}{2}}^2
  +\frac{\tau}{2}\gamma_{2}\|\tilde{e}_{v}^{\frac{1}{2}}\|^2
 =-\frac{\tau}{2}(\tilde{G}^{\frac{1}{2}},\tilde{e}_{v}^{\frac{1}{2}})
  +\frac{\tau}{2}(\tilde{Q}_{2}^{\frac{1}{2}},\tilde{e}_{v}^{\frac{1}{2}}), \label{tmerr2}\\
  &(\tilde{e}_{r}^{\frac{1}{2}})^{2}=\frac{\tau}{2}(\tilde{F}^{\frac{1}{2}},\tilde{e}_{r}^{\frac{1}{2}})
  +\frac{\tau}{2}(\tilde{Q}_{3}^{\frac{1}{2}},\tilde{e}_{r}^{\frac{1}{2}}).\label{tmerr3}
\end{align}
Using \eqref{tmdstay1}, the \eqref{tmerr2} can be rewritten as
\begin{equation}\label{tmerr4}
  \|\tilde{e}_{v}^{\frac{1}{2}}\|^2+\kappa |\tilde{e}_{u}^{\frac{1}{2}}|_{\frac{\a}{2}}^2
+\frac{\tau}{2}\gamma_{1}|\tilde{e}_{v}^{\frac{1}{2}}|_{\frac{\a}{2}}^2
  +\frac{\tau}{2}\gamma_{2}\|\tilde{e}_{v}^{\frac{1}{2}}\|^2
 =-\frac{\tau}{2}(\tilde{G}^{\frac{1}{2}},\tilde{e}_{v}^{\frac{1}{2}})
  +\frac{\tau}{2}(\tilde{Q}_{2}^{\frac{1}{2}},\tilde{e}_{v}^{\frac{1}{2}})+\frac{\tau}{2}
  \kappa ((-\Delta)^{\frac{\a}{2}}\tilde{e}_{u}^{\frac{1}{2}},\tilde{Q}_{1}^{\frac{1}{2}}).
\end{equation}
  In addition, $\tilde{G}^{\frac{1}{2}}$ and $\tilde{F}^{\frac{1}{2}}$ could be represented as
  \begin{align}
   &\tilde{G}^{\frac{1}{2}}
   =  \tilde{r}^{\frac{1}{2}}(\frac{F'(u^{0})}{\sqrt{E(u^{0})}}-\frac{F'(U^{0})}{\sqrt{E(U^{0})}})
   +\tilde{e}_{r}^{\frac{1}{2}}\frac{F'(U^{0})}{\sqrt{E(U^{0})}}, \label{tmdsnonerr} \\
  & \tilde{F}^{\frac{1}{2}}
 = \bigg(\frac{1}{2\sqrt{E(u^{0})}}F'(u^{0})-\frac{1}{2\sqrt{E(U^{0})}}F'(U^{0}),\frac{\tilde{U}^{\frac{1}{2}}-U^{0}}{\tau/2}\bigg)
  +\bigg(\frac{1}{2\sqrt{E(u^{0})}}F'(u^{0}),\frac{\tilde{e}_{u}^{\frac{1}{2}}}{\tau/2}\bigg). \label{tmdsnonerr1}
  \end{align}
  From \eqref{tmerr1}-\eqref{tmerr3} and Lemma \ref{le1.1}, using \eqref{tmerr4}-\eqref{tmdsnonerr1} and  using the Cauchy-Schwarz inequality and Young's inequality, we obtain
  \begin{align}
&\|\tilde{e}_{u}^{\frac{1}{2}}\|^2\leq C\tau\|\tilde{e}_{u}^{\frac{1}{2}}\|^2+C\tau \|\tilde{Q}_{1}^{\frac{1}{2}}\|^2+C\tau\|\tilde{e}_{v}^{\frac{1}{2}}\|^2 , \label{tmerr11}  \\
& \|\tilde{e}_{v}^{\frac{1}{2}}\|^2+\kappa |\tilde{e}_{u}^{\frac{1}{2}}|_{\frac{\a}{2}}^2
+\frac{\tau}{2}\gamma_{1}|\tilde{e}_{v}^{\frac{1}{2}}|_{\frac{\a}{2}}^2
  +\frac{\tau}{2}\gamma_{2}\|\tilde{e}_{v}^{\frac{1}{2}}\|^2
 \leq C\tau(\tilde{e}_{r}^{\frac{1}{2}})^2+C\tau\|\tilde{e}_{v}^{\frac{1}{2}}\|^2
 +C\tau\|\tilde{Q}_{2}^{\frac{1}{2}}\|^2  \nonumber  \\
 &~~~~~~~~~~~~~~~~~~~~~~~~~~~~~~~ ~~~~~~~~~~~~~~~~~~~~~~~~~~~~~+C\tau
 |\tilde{e}_{u}^{\frac{1}{2}}|_{\frac{\a}{2}}^{2}+C\tau|\tilde{Q}_{1}^{\frac{1}{2}}|_{\frac{\a}{2}}^{2}, \label{tmerr22} \\
  & (\tilde{e}_{r}^{\frac{1}{2}})^{2}\leq C\tau\|\tilde{e}_{v}^{\frac{1}{2}}\|^2+C\tau (\tilde{e}_{r}^{\frac{1}{2}})^2
  +C\tau\|\tilde{Q}_{3}^{\frac{1}{2}}\|^2+ C\tau \|\tilde{Q}_{1}^{\frac{1}{2}}\|^2.\label{tmerr33}
\end{align}
By combining the above equations, we can obtain
\begin{align}
 \|\tilde{e}_{u}^{\frac{1}{2}}\|^2+\kappa |\tilde{e}_{u}^{\frac{1}{2}}|_{\frac{\a}{2}}^2+\|\tilde{e}_{v}^{\frac{1}{2}}\|^2
 +(\tilde{e}_{r}^{\frac{1}{2}})^{2}\leq
 C\tau  \Big(|\tilde{Q}_{1}^{\frac{1}{2}}|_{\frac{\a}{2}}^{2}+\|\tilde{Q}_{2}^{\frac{1}{2}}\|^2+
 \|\tilde{Q}_{3}^{\frac{1}{2}}\|^2\Big)\leq C\tau^3. \label{timerrtheo1}
\end{align}
Moreover, take the inner product of \eqref{tmdstay2} by $(-\Delta)^{\frac{\a}{2}}\tilde{e}_{v}^{\frac{1}{2}}$ to get
\begin{align}\label{alfaer1}
 |\tilde{e}_{v}^{\frac{1}{2}}|_{\frac{\a}{2}}^2+\frac{\tau}{2}\kappa ((-\Delta)^{\frac{\a}{2}}\tilde{e}_{u}^{\frac{1}{2}},(-\Delta)^{\frac{\a}{2}}\tilde{e}_{v}^{\frac{1}{2}})
+\frac{\tau}{2}\gamma_{1}|\tilde{e}_{v}^{\frac{1}{2}}|_{\a}^2
  +\frac{\tau}{2}\gamma_{2}|\tilde{e}_{v}^{\frac{1}{2}}|_{\frac{\a}{2}}^2
 =-\frac{\tau}{2}(\tilde{G}^{\frac{1}{2}}
  +\tilde{Q}_{2}^{\frac{1}{2}},(-\Delta)^{\frac{\a}{2}}\tilde{e}_{v}^{\frac{1}{2}}).
\end{align}
Using $\frac{(-\Delta)^{\frac{\a}{2}}\tilde{e}_{u}^{\frac{1}{2}}}{\tau/2}=(-\Delta)^{\frac{\a}{2}}\tilde{e}_{v}^{\frac{1}{2}}+(-\Delta)^{\frac{\a}{2}}\tilde{Q}_{1}^{\frac{1}{2}}$, \eqref{alfaer1} can be rewritten as
\begin{align*}
 &|\tilde{e}_{v}^{\frac{1}{2}}|_{\frac{\a}{2}}^2+\kappa |\tilde{e}_{u}^{\frac{1}{2}}|_{\a}^{2}
+\frac{\tau}{2}\gamma_{1}|\tilde{e}_{v}^{\frac{1}{2}}|_{\a}^2
  +\frac{\tau}{2}\gamma_{2}|\tilde{e}_{v}^{\frac{1}{2}}|_{\frac{\a}{2}}^2 
 \leq \frac{\tau}{2}\bigg|(\tilde{G}^{\frac{1}{2}},(-\Delta)^{\frac{\a}{2}}\tilde{e}_{v}^{\frac{1}{2}})
  +(\tilde{Q}_{2}^{\frac{1}{2}},(-\Delta)^{\frac{\a}{2}}\tilde{e}_{v}^{\frac{1}{2}})
  +\kappa ((-\Delta)^{\frac{\a}{2}}\tilde{Q}_{1}^{\frac{1}{2}},(-\Delta)^{\frac{\a}{2}}\tilde{e}_{u}^{\frac{1}{2}})\bigg|.
\end{align*}
From Lemma \ref{le1.1}, the three terms on the right side of the above equation 
can be controlled by
\begin{align}
 & |(\tilde{G}^{\frac{1}{2}},(-\Delta)^{\frac{\a}{2}}\tilde{e}_{v}^{\frac{1}{2}})|
  \leq \Big|\tilde{e}_{r}^{\frac{1}{2}}(\frac{1}{\sqrt{E(U^{0})}}F'(U^{0}),(-\Delta)^{\frac{\a}{2}}\tilde{e}_{v}^{\frac{1}{2}})\Big|
  \leq C(L)(\tilde{e}_{r}^{\frac{1}{2}})^2+\frac{c_1}{2}|\tilde{e}_{v}^{\frac{1}{2}}|_{\frac{\a}{2}}^{2}, \label{tims2}\\
 & |(\tilde{Q}_{2}^{\frac{1}{2}},(-\Delta)^{\frac{\a}{2}}\tilde{e}_{v}^{\frac{1}{2}}) |
  \leq \frac{c_2}{2}|\tilde{e}_{v}^{\frac{1}{2}}|_{\frac{\a}{2}}^{2}+\frac{c_3}{2}|\tilde{Q}_{2}^{\frac{1}{2}}|_{\frac{\a}{2}}^{2},\label{tims3}\\
 & \Big|((-\Delta)^{\frac{\a}{2}}\tilde{Q}_{1}^{\frac{1}{2}},(-\Delta)^{\frac{\a}{2}}\tilde{e}_{u}^{\frac{1}{2}})\Big |
  \leq \frac{c_4}{2}|\tilde{e}_{u}^{\frac{1}{2}}|_{\a}^{2}+\frac{c_5}{2}|\tilde{Q}_{1}^{\frac{1}{2}}|_{\a}^{2}.\label{tims4}
\end{align}
From \eqref{timerrtheo1} and \eqref{tims2}-\eqref{tims4},
then we can get
\begin{equation}\label{timerrtheo2}
 \|\tilde{e}_{u}^{\frac{1}{2}}\|^2+\kappa |\tilde{e}_{u}^{\frac{1}{2}}|_{\frac{\a}{2}}^2+\|\tilde{e}_{v}^{\frac{1}{2}}\|^2
 +(\tilde{e}_{r}^{\frac{1}{2}})^{2}+ |\tilde{e}_{v}^{\frac{1}{2}}|_{\frac{\a}{2}}^2+\kappa |\tilde{e}_{u}^{\frac{1}{2}}|_{\a}^{2} \leq C\tau^3.
\end{equation}
By virtue of Lemma \ref{le1.3}, \eqref{timerrtheo1} and \eqref{timerrtheo2}, we have
\begin{equation}\label{timeboutheo1}
  \|\tilde{U}^{\frac{1}{2}}\|_{\infty}\leq  \|\tilde{u}^{\frac{1}{2}}\|_{\infty}+\|\tilde{e}_{u}^{\frac{1}{2}}\|_{\infty}\leq \|\tilde{u}^{\frac{1}{2}}\|_{\infty}+C|\tilde{e}_{u}^{\frac{1}{2}}|_{\a} \leq 1+L\leq M,
\end{equation}
where $\tau \leq C^{-\frac{2}{3}}$.

2. The second step: estimate $e_{u}^{1}$, $e_{v}^{1}$ and $e_{r}^{1}$.

Subtract \eqref{sinets1}-\eqref{sinets3} from \eqref{sinetay1}-\eqref{sinetay3} to obtain for $n=0$.
\begin{eqnarray}
 &&\delta_{t}e_{u}^{\frac{1}{2}}=\bar{e}_{v}^{\frac{1}{2}}+Q_{1}^{\frac{1}{2}}, \label{1stmerr1}\\
  &&\delta_{t}e_{v}^{\frac{1}{2}}+\kappa (-\Delta)^{\frac{\a}{2}}\bar{e}_{u}^{\frac{1}{2}}+\gamma_{1}(-\Delta)^{\frac{\a}{2}}\bar{e}_{v}^{\frac{1}{2}}+\gamma_{2}\bar{e}_{v}^{\frac{1}{2}}
  +G^{\frac{1}{2}} =Q_{2}^{\frac{1}{2}} ,  \label{1stmerr2}\\
  &&\delta_{t}e_{r}^{\frac{1}{2}}=F^{\frac{1}{2}} +Q_{3}^{\frac{1}{2}},\label{1stmerr3}
\end{eqnarray}
where $$G^{\frac{1}{2}}=r^{\frac{1}{2}}\frac{F'(\tilde{u}^{\frac{1}{2}})}{\sqrt{E(\tilde{u}^{\frac{1}{2}})}}
-R^{\frac{1}{2}}\frac{F'(\tilde{U}^{\frac{1}{2}})}{\sqrt{E(\tilde{U}^{\frac{1}{2}})}},$$ $$F^{\frac{1}{2}}=\frac{1}{2\sqrt{E(\tilde{u}^{\frac{1}{2}})}}(F'(\tilde{u}^{\frac{1}{2}}),\delta_{t}u^{\frac{1}{2}})-
\frac{1}{2\sqrt{E(\tilde{U}^{\frac{1}{2}})}}(F'(\tilde{U}^{\frac{1}{2}}),\delta_{t}U^{\frac{1}{2}}) .$$
It is obvious that
\begin{small}
\begin{eqnarray}
&&  G^{\frac{1}{2}}=r^{\frac{1}{2}}(\frac{F'(\tilde{u}^{\frac{1}{2}})}{\sqrt{E(\tilde{u}^{\frac{1}{2}})}}
  -\frac{F'(\tilde{U}^{\frac{1}{2}})}{\sqrt{E(\tilde{U}^{\frac{1}{2}})}})+
\bar{e}_{r}^{\frac{1}{2}}\frac{F'(\tilde{U}^{\frac{1}{2}})}{\sqrt{E(\tilde{U}^{\frac{1}{2}})}}, \nonumber\\ 
&&F^{\frac{1}{2}}=(\frac{1}{2\sqrt{E(\tilde{u}^{\frac{1}{2}})}}F'(\tilde{u}^{\frac{1}{2}})
-\frac{1}{2\sqrt{E(\tilde{U}^{\frac{1}{2}})}}F'(\tilde{U}^{\frac{1}{2}}),\delta_{t}u^{\frac{1}{2}})
+(\frac{1}{2\sqrt{E(\tilde{U}^{\frac{1}{2}})}}F'(\tilde{U}^{\frac{1}{2}}),\delta_{t}e_{u}^{\frac{1}{2}}),\nonumber\\ 
&&\frac{1}{\sqrt{E(\tilde{u}^{\frac{1}{2}})}}F'(\tilde{u}^{\frac{1}{2}})
-\frac{1}{\sqrt{E(\tilde{U}^{\frac{1}{2}})}}F'(\tilde{U}^{\frac{1}{2}})=
\frac{F'(\tilde{u}^{\frac{1}{2}})-F'(\tilde{U}^{\frac{1}{2}})}{\sqrt{E(\tilde{u}^{\frac{1}{2}})}}+
\frac{F'(\tilde{U}^{\frac{1}{2}})(E(\tilde{u}^{\frac{1}{2}})
-E(\tilde{U}^{\frac{1}{2}}))}
{\sqrt{E(\tilde{U}^{\frac{1}{2}})}\sqrt{E(\tilde{u}^{\frac{1}{2}})}
(\sqrt{E(\tilde{u}^{\frac{1}{2}})}+\sqrt{E(\tilde{U}^{\frac{1}{2}})})}. \nonumber 
\end{eqnarray}
\end{small}
Since $e_{u}^{0}=0$, $e_{v}^{0}=0$, and $e_{r}^{0}=0$, then taking inner product of $2\tau e_{u}^{1}$ and $2\tau e_{v}^{1}$ in \eqref{1stmerr1}-\eqref{1stmerr2}, and   multiplying $2\tau e_{r}^{1}$ on the both sides of \eqref{1stmerr3}, we have
\begin{eqnarray}
&&2\|e_{u}^{1}\|^{2}=\tau(e_{v}^{1},e_{u}^{1})+2\tau(Q_{1}^{\frac{1}{2}},e_{u}^{1}), \label{1stmerrn1}\\
  &&2\|e_{v}^{1}\|^{2}+\kappa \tau((-\Delta)^{\frac{\a}{2}}e_{u}^{1},e_{v}^{1})+\tau \gamma_{1}((-\Delta)^{\frac{\a}{2}}e_{v}^{1},e_{v}^{1})+\gamma_{2}\tau (e_{v}^{1},e_{v}^{1})
  +2\tau (G^{\frac{1}{2}},e_{v}^{1}) =2\tau(Q_{2}^{\frac{1}{2}},e_{v}^{1}) ,  \label{1stmerrn2}\\
  &&(e_{r}^{1})^{2}=\tau(F^{\frac{1}{2}},e_{r}^{1}) +\tau(Q_{3}^{\frac{1}{2}},e_{r}^{1}).\label{1stmerrn3}
\end{eqnarray}

By using the Cauchy-Schwarz inequality and Young's inequality, we obtain
\begin{align}
 &\|e_{u}^{1}\|^{2}\leq C\tau \|e_{v}^{1}\|^2+C\tau \|e_{u}^{1}\|^{2}+C\tau \|Q_{1}^{\frac{1}{2}}\|^{2}. \label{1sterr1}\\
 &2\|e_{v}^{1}\|^{2}+2\kappa |e_{u}^{1}|_{\frac{\a}{2}}^{2}+\tau \gamma_{1}|e_{v}^{1}|_{\frac{\a}{2}}^{2}+ \gamma_{2}\tau \|e_{v}^{1}\|^2 \leq C\tau \bigg(|(G^{\frac{1}{2}}, e_{v}^{1})|+  \|e_{v}^{1}\|^2+ \|Q_{2}^{\frac{1}{2}}\|^2+|((-\Delta)^{\frac{\a}{4}}e_{u}^{1},(-\Delta)^{\frac{\a}{4}}Q_{1}^{\frac{1}{2}})|\bigg) \nonumber\\
 &~~~~~~~~~~~~~~~~~~~~~~~~~~~~~~~~~~~~~~~~~~~\leq C\tau |(G^{\frac{1}{2}}, e_{v}^{1})|+ C\tau \|e_{v}^{1}\|^2+C\tau \|Q_{2}^{\frac{1}{2}}\|^2+C\tau|e_{u}^{1}|_{\frac{\a}{2}}^2+C\tau|Q_{1}^{\frac{1}{2}}|_{\frac{\a}{2}}^2, \label{1sterr2}\\
&(e_{r}^{1})^{2}\leq C\tau |(F^{\frac{1}{2}},e_{r}^{1})| + C\tau \|e_{r}^{1}\|^{2} +C\tau \|Q_{3}^{\frac{1}{2}}\|^2. \label{1sterr3}
\end{align}
Since $|\tilde{U}^{\frac{1}{2}}|_{\infty}$ is bounded,
\begin{align}
&|(G^{\frac{1}{2}},e_{v}^{1})| \leq  \bigg(|r^{\frac{1}{2}}|_{\infty}|\tilde{e}_{u}^{\frac{1}{2}}|+
|e_{r}^{\frac{1}{2}}|\Big\|\frac{F'(\tilde{U}^{\frac{1}{2}})}{\sqrt{E(\tilde{U}^{\frac{1}{2}})}}\Big\|_{\infty}\bigg)|e_{v}^{1}|\leq C\|\tilde{e}_{u}^{\frac{1}{2}}\|^2+C(e_{r}^{\frac{1}{2}})^2+C\|e_{v}^{1}\|^2, \label{1strr1}\\
&|(F^{\frac{1}{2}},e_{r}^{1})| \leq C|e_{r}^{1}|\bigg|\Big(\frac{1}{2\sqrt{E(\tilde{u}^{\frac{1}{2}})}}F'(\tilde{u}^{\frac{1}{2}})
-\frac{1}{2\sqrt{E(\tilde{U}^{\frac{1}{2}})}}F'(\tilde{U}^{\frac{1}{2}}),\frac{u^{1}-u^{0}}{\tau}\Big)\bigg|
+|e_{r}^{1}|\bigg|(\frac{1}{2\sqrt{E(\tilde{U}^{\frac{1}{2}})}}F'(\tilde{U}^{\frac{1}{2}}),\delta_{t}e_{u}^{\frac{1}{2}})\bigg|. \nonumber
\end{align}
where $$u^{1}-u^{0}=\tau u_{t}({t_{\frac{1}{2}}})-\frac{1}{2}\bigg(\int_{t_{0}}^{t_{\frac{1}{2}}}(t_{0}-s)^2\frac{\partial^{3} u}{\partial t^{3}}(s)ds+\int_{t_{1}}^{t_{\frac{1}{2}}}(t_{1}-s)^2\frac{\partial^{3} u}{\partial t^{3}}(s)ds\bigg).$$ From the above equations and \eqref{1stmerr1}, using Cauchy-Schwarz inequality and Young's inequality,
Thus,
\begin{align}\label{er1}
\tau|(F^{\frac{1}{2}},e_{r}^{1})| \leq C \tau |u_{t}({t_{\frac{1}{2}}})|_{\infty}((e_{r}^{1})^2+\|\tilde{e}_{u}^{\frac{1}{2}}\|^2)+C\tau^{4}\int_{t_{0}}^{t_{1}}\|\frac{\partial^{3}u}{\partial t^{3}}(s)\|^2ds
+C\tau(e_{r}^{1})^2+C\tau \|e_{v}^{1}\|^2 +C \tau \|Q_{1}^{\frac{1}{2}}\|^2.
\end{align}
Combining \eqref{timerrtheo2} and \eqref{1sterr1}-\eqref{er1},
we  get
{\color{blue}
\begin{align}\label{er2st1}
&\|e_{u}^{1}\|^{2}+ \|e_{v}^{1}\|^{2}+\kappa |e_{u}^{1}|_{\frac{\a}{2}}^{2}+(e_{r}^{1})^{2}\leq
C\tau^{4}\int_{t_{0}}^{t_{1}}\|\frac{\partial^{3}u}{\partial t^{3}}(s)\|^2ds+  \nonumber \\
&~~~~~~~~~~~~~~~~C \tau ( \|Q_{2}^{\frac{1}{2}}\|^2+ |Q_{1}^{\frac{1}{2}}|_{\frac{\a}{2}}^2+ \|Q_{3}^{\frac{1}{2}}\|^2)+C\tau\|\tilde{e}_{u}^{\frac{1}{2}}\|^2 \leq C\tau^4.
\end{align}
}
 Applying $(-\Delta)^{\frac{\a}{2}}$ on the both sides of \eqref{1stmerr1}, taking the inner product of the resulting equation with $2\tau (-\Delta)^{\frac{\a}{2}}e_{u}^{1}$, one arrives at
 \begin{equation}\label{2stalp}
   2|e_{u}^{1}|_{\a}^{2}=\tau((-\Delta)^{\frac{\a}{2}}e_{v}^{1},(-\Delta)^{\frac{\a}{2}}e_{u}^{1})
   +2\tau((-\Delta)^{\frac{\a}{2}}Q_{1}^{\frac{1}{2}},(-\Delta)^{\frac{\a}{2}}e_{u}^{1}).
 \end{equation}
 Taking the inner product of \eqref{1stmerr2}  with $2\tau (-\Delta)^{\frac{\a}{2}}e_{v}^{1}$ and using \eqref{2stalp}, we obtain
\begin{align}\label{2stalp1}
2|e_{v}^{1}|_{\frac{\a}{2}}^{2}+2\kappa|e_{u}^{1}|_{\a}^{2} -2\kappa\tau((-\Delta)^{\frac{\a}{2}}Q_{1}^{\frac{1}{2}},(-\Delta)^{\frac{\a}{2}}e_{u}^{1})
+\gamma_{1}|e_{v}^{1}|_{\a}^{2}+ \nonumber \\
\gamma_{2}\tau |e_{v}^{1}|_{\frac{\a}{2}}^{2}+2\tau(G^{\frac{1}{2}},(-\Delta)^{\frac{\a}{2}}e_{v}^{1})
=2\tau(Q_{2}^{\frac{1}{2}},(-\Delta)^{\frac{\a}{2}}e_{v}^{1}).
\end{align}
By the Cauchy-Schwarz inequality and Young's inequality,
it leads to from the equation \eqref{2stalp1}
\begin{equation}\label{2stalp1}
2|e_{v}^{1}|_{\frac{\a}{2}}^{2}+2\kappa|e_{u}^{1}|_{\a}^{2}
\leq C\tau\|(-\Delta)^{\frac{\a}{4}}G^{\frac{1}{2}}\|^2+C\tau |e_{v}^{1}|_{\frac{\a}{2}}^{2}+
C\tau\|(-\Delta)^{\frac{\a}{4}}Q_{2}^{\frac{1}{2}}\|^{2}
+C\tau\|(-\Delta)^{\frac{\a}{2}}Q_{1}^{\frac{1}{2}}\|^2+C\tau|e_{u}^{1}|_{\a}^{2}.
\end{equation}
Meanwhile, we have
 \begin{small}
\begin{align}
 &\|(-\Delta)^{\frac{\a}{4}}G^{\frac{1}{2}}\|^2=\bigg\|(-\Delta)^{\frac{\a}{4}}\Big(r^{\frac{1}{2}}
 \Big(\frac{F'(\tilde{u}^{\frac{1}{2}})}{\sqrt{E(\tilde{u}^{\frac{1}{2}})}}
  -\frac{F'(\tilde{U}^{\frac{1}{2}})}{\sqrt{E(\tilde{U}^{\frac{1}{2}})}}\Big)+
\bar{e}_{r}^{\frac{1}{2}}\frac{F'(\tilde{U}^{\frac{1}{2}})}{\sqrt{E(\tilde{U}^{\frac{1}{2}})}}\Big)\bigg\|^2 \nonumber \\
&\leq |r^{\frac{1}{2}}|_{\infty}^2\bigg\|(-\Delta)^{\frac{\a}{4}}
\Big(\frac{F'(\tilde{u}^{\frac{1}{2}})}{\sqrt{E(\tilde{u}^{\frac{1}{2}})}}
  -\frac{F'(\tilde{U}^{\frac{1}{2}})}{\sqrt{E(\tilde{U}^{\frac{1}{2}})}}\Big)\bigg\|^2
  +(e_{r}^{1})^2\Big\|(-\Delta)^{\frac{\a}{4}}
  \frac{F'(\tilde{U}^{\frac{1}{2}})}{\sqrt{E(\tilde{U}^{\frac{1}{2}})}}\Big\|^2, \label{midie1}
\end{align}
\end{small}
 and  by  using Lemma \ref{le1.2}, we obtain
 {\color{blue}
\begin{align}\label{midie2}
&\bigg\|(-\Delta)^{\frac{\a}{4}}
\Big(\frac{F'(\tilde{u}^{\frac{1}{2}})}{\sqrt{E(\tilde{u}^{\frac{1}{2}})}}
  -\frac{F'(\tilde{U}^{\frac{1}{2}})}{\sqrt{E(\tilde{U}^{\frac{1}{2}})}}\Big)\bigg\|^2\leq \bigg\|(-\Delta)^{\frac{\a}{4}}\frac{F'(\tilde{U}^{\frac{1}{2}})(E(\tilde{u}^{\frac{1}{2}})
-E(\tilde{U}^{\frac{1}{2}}))}
{\sqrt{E(\tilde{U}^{\frac{1}{2}})}\sqrt{E(\tilde{u}^{\frac{1}{2}})}
\Big(\sqrt{E(\tilde{u}^{\frac{1}{2}})}+\sqrt{E(\tilde{U}^{\frac{1}{2}})}\Big)}\bigg\|^2 \nonumber\\
&+ \bigg\|(-\Delta)^{\frac{\a}{4}}\frac{F'(\tilde{u}^{\frac{1}{2}})-F'(\tilde{U}^{\frac{1}{2}})}{\sqrt{E(\tilde{u}^{\frac{1}{2}})}}
\bigg\|^2
 \leq \bigg\|(-\Delta)^{\frac{\a}{4}}\frac{\partial_{u}F'(\tilde{\xi}) (\tilde{u}^{\frac{1}{2}}-\tilde{U}^{\frac{1}{2}})}{\sqrt{E(\tilde{u}^{\frac{1}{2}})}}\bigg\|^2+C|E(\tilde{u}^{\frac{1}{2}})
-E(\tilde{U}^{\frac{1}{2}})|^2\bigg\|(-\Delta)^{\frac{\a}{4}}F'(\tilde{U}^{\frac{1}{2}})\bigg\|^2\nonumber\\
&\leq C|\tilde{e}_{u}^{\frac{1}{2}}|_{\frac{\a}{2}}^2.
\end{align}
}
Then  substituting \eqref{midie2} into \eqref{midie1} and using Lemma \ref{le1.2}, it follows from \eqref{2stalp1} that
\begin{equation}\label{1stg1}
 \|(-\Delta)^{\frac{\a}{4}}G^{\frac{1}{2}}\|^2\leq C|\tilde{e}_{u}^{\frac{1}{2}}|_{\frac{\a}{2}}^2+C(e_{r}^{1})^2.
\end{equation}
Combining \eqref{er2st1}, \eqref{2stalp1} and \eqref{1stg1}, we arrive at
{\color{blue}
\begin{align}\label{2}
 |e_{v}^{1}|_{\frac{\a}{2}}^{2}+\kappa|e_{u}^{1}|_{\a}^{2}+ \|e_{u}^{1}\|^{2}+ \|e_{v}^{1}\|^{2}+\kappa |e_{u}^{1}|_{\frac{\a}{2}}^{2}+(e_{r}^{1})^{2}
  \leq C\tau^4.
\end{align}
}
By virtue of Lemma \ref{le1.3} and \eqref{2}, we have
\begin{equation}\label{timeboutheo1}
  \|U^{1}\|_{\infty}\leq  \|u^{1}\|_{\infty}+\|e_{u}^{1}\|_{\infty}\leq \|u^{1}\|_{\infty}+C|e_{u}^{1}|_{\a} \leq 1+L\leq M,
\end{equation}
where $\tau \leq C^{-\frac{1}{2}}$.

 3. The third  step : we intend to use the mathematical induction method to estimate \eqref{errtheo1} holds  for $1\leq n \leq K$.

 First, we assume that \eqref{errtheo1} holds for $n\leq k$ and we intend to prove  that it also holds for $n=k+1$. Similar to \eqref{timeboutheo1}, we obtain
 {\color{blue}
 \begin{equation}\label{3}
  \|U^{n}\|_{\infty}\leq  M, ~1\leq n \leq k.
 \end{equation}
 }
 Subtract \eqref{sinets1}-\eqref{sinets3} from \eqref{sinetay1}-\eqref{sinetay3} to obtain the following error {\color{blue} equations}.
\begin{eqnarray}
 &&\delta_{t}e_{u}^{k+\frac{1}{2}}=\bar{e}_{v}^{k+\frac{1}{2}}+Q_{1}^{k+\frac{1}{2}}, \label{4}\\
  &&\delta_{t}e_{v}^{k+\frac{1}{2}}+\kappa (-\Delta)^{\frac{\a}{2}}\bar{e}_{u}^{k+\frac{1}{2}}+\gamma_{1}(-\Delta)^{\frac{\a}{2}}\bar{e}_{v}^{k+\frac{1}{2}}+
  \gamma_{2}\bar{e}_{v}^{k+\frac{1}{2}}
  +G^{k+\frac{1}{2}} =Q_{2}^{k+\frac{1}{2}} ,  \label{5}\\
  &&\delta_{t}e_{r}^{k+\frac{1}{2}}=F^{k+\frac{1}{2}} +Q_{3}^{k+\frac{1}{2}}.\label{6}
\end{eqnarray}
Taking the inner product of $2\tau \bar{e}_{u}^{k+\frac{1}{2}}$ and $2\tau \bar{e}_{v}^{k+\frac{1}{2}}$ in \eqref{4}-\eqref{5} respectively,  as well as multiplying $2\tau \bar{e}_{r}^{k+\frac{1}{2}}$ on the both sides of \eqref{6}, we have
{\color{blue}
\begin{eqnarray}
&& \|e_{u}^{k+1}\|^2-\|e_{u}^{k}\|^2=2\tau(\bar{e}_{v}^{k+\frac{1}{2}},\bar{e}_{u}^{k+\frac{1}{2}})
+2\tau(Q_{1}^{k+\frac{1}{2}},\bar{e}_{u}^{k+\frac{1}{2}}), \label{7}\\
  &&\|e_{v}^{k+1}\|^2-\|e_{v}^{k}\|^2+\kappa (|e_{u}^{k+1}|_{\frac{\a}{2}}^{2}-|e_{u}^{k}|_{\frac{\a}{2}}^{2})
  +2\tau \gamma_{1}|\bar{e}_{v}^{k+\frac{1}{2}}|_{\frac{\a}{2}}^{2}+2\gamma_{2}\tau \|\bar{e}_{v}^{k+\frac{1}{2}}\|^2 \nonumber \\
  &&=2\tau(Q_{2}^{k+\frac{1}{2}},\bar{e}_{v}^{k+\frac{1}{2}})-2\tau (G^{k+\frac{1}{2}},\bar{e}_{v}^{k+\frac{1}{2}}) +2\tau \kappa(Q_{1}^{k+\frac{1}{2}},(-\Delta)^{\a/2}\bar{e}_{u}^{k+\frac{1}{2}}), \label{8}\\
  &&(e_{r}^{k+1})^{2}-(e_{r}^{k})^{2}=2\tau(F^{k+\frac{1}{2}},\bar{e}_{r}^{k+\frac{1}{2}}) +2\tau(Q_{3}^{k+\frac{1}{2}},\bar{e}_{r}^{k+\frac{1}{2}}),\label{9}
\end{eqnarray}
}
where
\begin{equation*}
  G^{k+\frac{1}{2}}=r^{k+\frac{1}{2}}\bigg(\frac{F'(\tilde{u}^{k+\frac{1}{2}})}{\sqrt{E(\tilde{u}^{k+\frac{1}{2}})}}
  -\frac{F'(\tilde{U}^{k+\frac{1}{2}})}{\sqrt{E(\tilde{U}^{k+\frac{1}{2}})}}\bigg)+
\bar{e}_{r}^{k+\frac{1}{2}}\frac{F'(\tilde{U}^{k+\frac{1}{2}})}{\sqrt{E(\tilde{U}^{k+\frac{1}{2}})}},
\end{equation*}
and
\begin{small}
\begin{eqnarray}
&&F^{k+\frac{1}{2}}=\bigg(\frac{1}{2\sqrt{E(\tilde{u}^{k+\frac{1}{2}})}}F'(\tilde{u}^{k+\frac{1}{2}})
-\frac{1}{2\sqrt{E(\tilde{U}^{k+\frac{1}{2}})}}F'(\tilde{U}^{k+\frac{1}{2}}),\delta_{t}u^{k+\frac{1}{2}}\bigg)
+\Big(\frac{1}{2\sqrt{E(\tilde{U}^{k+\frac{1}{2}})}}F'(\tilde{U}^{k+\frac{1}{2}}),\delta_{t}e_{u}^{k+\frac{1}{2}}\Big), \nonumber 
\end{eqnarray}
\end{small}
where
\begin{small}
\begin{eqnarray}
&&\frac{1}{\sqrt{E(\tilde{u}^{k+\frac{1}{2}})}}F'(\tilde{u}^{k+\frac{1}{2}})
-\frac{1}{\sqrt{E(\tilde{U}^{k+\frac{1}{2}})}}F'(\tilde{U}^{k+\frac{1}{2}})= \nonumber\\
&& \frac{F'(\tilde{u}^{k+\frac{1}{2}})-F'(\tilde{U}^{k+\frac{1}{2}})}{\sqrt{E(\tilde{u}^{k+\frac{1}{2}})}}+\frac{F'(\tilde{U}^{k+\frac{1}{2}})(E(\tilde{u}^{k+\frac{1}{2}})
-E(\tilde{U}^{k+\frac{1}{2}}))}
{\sqrt{E(\tilde{U}^{k+\frac{1}{2}})}\sqrt{E(\tilde{u}^{k+\frac{1}{2}})}
\bigg(\sqrt{E(\tilde{u}^{k+\frac{1}{2}})}+\sqrt{E(\tilde{U}^{k+\frac{1}{2}})}\bigg)}, \label{tay1} \\
&&u^{k+1}-u^{k}=\tau u_{t}({t_{k+\frac{1}{2}}})-\frac{1}{2}\bigg(\int_{t_{k}}^{t_{k+\frac{1}{2}}}(t_{k}-s)^2\frac{\partial^{3} u}{\partial t^{3}}(s)ds+\int_{t_{k+1}}^{t_{k+\frac{1}{2}}}(t_{k+1}-s)^2\frac{\partial^{3} u}{\partial t^{3}}(s)ds\bigg).\label{tay2}
\end{eqnarray}
\end{small}
Similar to the proof for \eqref{1strr1} and \eqref{er1}, we have the following {\color{blue}estimates}
\begin{align}
&|(G^{k+\frac{1}{2}},\bar{e}_{v}^{k+\frac{1}{2}})|\leq \bigg(|\bar{r}^{k+\frac{1}{2}}|_{\infty}|\tilde{e}_{u}^{k+\frac{1}{2}}|+
|\bar{e}_{r}^{k+\frac{1}{2}}|\|\frac{F'(\tilde{U}^{k+\frac{1}{2}})}{\sqrt{E(\tilde{U}^{k+\frac{1}{2}})}}
\|_{\infty}\bigg)|\bar{e}_{v}^{k+\frac{1}{2}}| \nonumber \\
 &~~~~~~~~~~~~~~~\leq C\bigg(\|e_{u}^{k}\|^2+\|e_{u}^{k-1}\|^2+\|e_{r}^{k+1}\|^2+\|e_{v}^{k+1}\|^2+\|e_{r}^{k}\|^2+\|e_{v}^{k}\|^2\bigg), \label{13}\\
&\tau|(F^{k+\frac{1}{2}},\bar{e}_{r}^{k+\frac{1}{2}})|\leq C \tau |u_{t}({t_{k+\frac{1}{2}}})|_{\infty}\bigg((e_{r}^{k+1})^2+(e_{r}^{k})^2+\|e_{u}^{k}\|^2+\|e_{u}^{k-1}\|^2\bigg)+ \nonumber \\
&~~~~~~~~~~~~~~~C\tau^{4}\int_{t_{k}}^{t_{k+1}}\|\frac{\partial^{3}u}{\partial t^{3}}(s)\|^2ds +C \tau\bigg( \|e_{v}^{k+1}\|^2+ \|e_{v}^{k}\|^2+ \|Q_{1}^{k+\frac{1}{2}}\|^2\bigg).\label{14}
\end{align}
Combining  \eqref{13}-\eqref{14}, using Cauchy-Schwarz inequality and Young's inequality, we get
{\color{blue}
\begin{align}\label{15}
 &\|e_{u}^{k+1}\|^2-\|e_{u}^{k}\|^2+\|e_{v}^{k+1}\|^2-\|e_{v}^{k}\|^2+\kappa (|e_{u}^{k+1}|_{\frac{\a}{2}}^{2}-|e_{u}^{k}|_{\frac{\a}{2}}^{2})+(e_{r}^{k+1})^2-(e_{r}^{k})^2 \nonumber \\
&~~~~~~~~~~~~ \leq C\tau\bigg(\|e_{v}^{k+1}\|^2+ \|e_{v}^{k}\|^2+\|e_{u}^{k}\|^2+\|e_{u}^{k-1}\|^2+(e_{r}^{k+1})^2+(e_{r}^{k})^2+|e_{u}^{k+1}|_{\frac{\a}{2}}^{2}+|e_{u}^{k}|_{\frac{\a}{2}}^{2}\bigg)+\nonumber \\
 &~~~~~~~~~~~C\tau\bigg(\|Q_{1}^{k+\frac{1}{2}}\|^2+|Q_{1}^{k+\frac{1}{2}}|_{\frac{\a}{2}}^2+\|Q_{2}^{k+\frac{1}{2}}\|^2+\|Q_{3}^{k+\frac{1}{2}}\|^2\bigg)
 +C\tau^{4}\int_{t_{k}}^{t_{k+1}}\|\frac{\partial^{3}u}{\partial t^{3}}(s)\|^2ds.
\end{align}
}
 Applying $(-\Delta)^{\frac{\a}{2}}$ on the both sides of \eqref{4}, and then taking the inner product of the resulting equation with $\tau(-\Delta)^{\frac{\a}{2}}\bar{e}_{u}^{k+\frac{1}{2}}$, one arrives at
 \begin{equation}\label{16}
   |e_{u}^{k+1}|_{\a}^{2}-|e_{u}^{k}|_{\a}^{2}=\tau((-\Delta)^{\frac{\a}{2}}\bar{e}_{v}^{k+\frac{1}{2}},(-\Delta)^{\frac{\a}{2}}\bar{e}_{u}^{k+\frac{1}{2}})
   +\tau((-\Delta)^{\frac{\a}{2}}Q_{1}^{k+\frac{1}{2}},(-\Delta)^{\frac{\a}{2}}\bar{e}_{u}^{k+\frac{1}{2}}),
 \end{equation}
and  taking the inner product of \eqref{5} with $2\tau (-\Delta)^{\frac{\a}{2}}\bar{e}_{v}^{k+\frac{1}{2}}$, we get
\begin{align}\label{17}
&|e_{v}^{k+1}|_{\frac{\a}{2}}^{2}-|e_{v}^{k}|_{\frac{\a}{2}}^{2}+\kappa|e_{u}^{k+1}|_{\a}^{2} -\kappa|e_{u}^{k}|_{\a}^{2}
+\tau\gamma_{1}|\bar{e}_{v}^{k+\frac{1}{2}}|_{\a}^{2}+
\gamma_{2}\tau |\bar{e}_{v}^{k+\frac{1}{2}}|_{\frac{\a}{2}}^{2} \nonumber \\
&=-\tau(G^{k+\frac{1}{2}},(-\Delta)^{\frac{\a}{2}}\bar{e}_{v}^{k+\frac{1}{2}})
+\tau(Q_{2}^{k+\frac{1}{2}},(-\Delta)^{\frac{\a}{2}}\bar{e}_{v}^{k+\frac{1}{2}})+
\kappa\tau\Big((-\Delta)^{\frac{\a}{2}}Q_{1}^{k+\frac{1}{2}},(-\Delta)^{\frac{\a}{2}}\bar{e}_{u}^{k+\frac{1}{2}}\Big).
\end{align}
Using Lemmas \ref{le1.1} and \ref{le1.2}, one  obtains
\begin{equation}\label{18}
 \|(-\Delta)^{\frac{\a}{4}} G^{k+\frac{1}{2}}\|^2\leq C\tau\big(|e_{u}^{k}|_{\frac{\a}{2}}^{2}+|e_{u}^{k-1}|_{\frac{\a}{2}}^{2}+(e_{r}^{k+1})^2+(e_{r}^{k})^2\big),
\end{equation}
then from \eqref{17}-\eqref{18}, we get
{\color{blue}
\begin{align}
|e_{v}^{k+1}|_{\frac{\a}{2}}^{2}-|e_{v}^{k}|_{\frac{\a}{2}}^{2}+\kappa|e_{u}^{k+1}|_{\a}^{2} -\kappa|e_{u}^{k}|_{\a}^{2}
\leq C\tau \big(|e_{v}^{k+1}|_{\a}^{2}+|e_{v}^{k}|_{\a}^{2}+|e_{u}^{k+1}|_{\a}^{2}+|e_{u}^{k}|_{\a}^{2}+\nonumber\\
|e_{u}^{k}|_{\a/2}^{2}+|e_{u}^{k-1}|_{\a/2}^{2}+(e_{r}^{k+1})^2+(e_{r}^{k})^2\big)
+\|(-\Delta)^{\frac{\a}{2}}Q_{1}^{k+\frac{1}{2}}\|^{2}+\|(-\Delta)^{\frac{\a}{4}}Q_{2}^{k+\frac{1}{2}}\|^{2}.\label{19}
\end{align}
}
Combining \eqref{15} and \eqref{19}, summing from $n=1$ to $n=k$, and then using Gronwall inequality, we get
\begin{align}\label{20}
 \|e_{u}^{k+1}\|^2+\|e_{v}^{k+1}\|^2+\kappa |e_{v}^{k+1}|_{\frac{\a}{2}}^{2}+\kappa |e_{u}^{k+1}|_{\frac{\a}{2}}^{2}+\kappa |e_{u}^{k+1}|_{\a}^{2}+(e_{r}^{k+1})^2\leq C\tau^4,
\end{align}
which implies $|e_{u}^{k+1}|_{\a} \leq C\tau^2$. Moreover, we arrive at
\begin{equation}\label{21}
  \|U^{k+1}\|_{\infty}\leq  \|u^{k+1}\|_{\infty}+ \|e_{u}^{k+1}\|_{\infty}\leq L+|e_{u}^{k+1}|_{\a}\leq M.
\end{equation}
\eqref{20} shows that  \eqref{errtheo1} holds  for $n=k+1$. We have completed the mathematical induction and thus the proof is completed.
\end{proof}

\section{Spatial error analysis}
In this section, we give the $\tau$-independent convergence results for the discrete full scheme. It implies that we do not need to assume the global Lipshitz condition. Similar to the technique in \cite{MR4045247}, from Lemma \ref{le1.3} and Lemma \ref{le1.5},
we have $\|P_{N}w\|_{\infty}\leq C\|w\|_{\a}$. Then we have the boundedness of $P_{N}U^{n}$ in $L_{\infty}$ norm, and 
denote $M_{2}=\max_{0\leq n \leq K}\{P_{N}U^{n}\}+1$.
For convenience, we denote
\begin{eqnarray}
 &&\bar{ e}_{N,u}^{n}=U^{n}-u_{N}^{n}=U^{n}-P_{N}U^{n}+P_{N}U^{n}-u_{N}^{n}=\rho_{u}^{n}+\theta_{u}^{n}, \nonumber\\
 &&\bar{ e}_{N,v}^{n}=V^{n}-v_{N}^{n}=V^{n}-P_{N}V^{n}+P_{N}V^{n}-v_{N}^{n}=\rho_{v}^{n}+\theta_{v}^{n}.
\end{eqnarray}
Combining the time-discrete system \eqref{sinets1}-\eqref{sinets3} and \eqref{22}-\eqref{27}, we directly get the following  error equations
\begin{align}
&(\delta_{t}\theta_{u}^{n+\frac{1}{2}},\psi)=(\bar{\theta}_{v}^{n+\frac{1}{2}},\psi), \label{28}\\
  &(\delta_{t}\theta_{v}^{n+\frac{1}{2}},\varphi)+\kappa ((-\Delta)^{\frac{\a}{2}}\bar{\theta}_{u}^{n+\frac{1}{2}},\varphi)+\gamma_{1}((-\Delta)^{\frac{\a}{2}}\bar{\theta}_{v}^{n+\frac{1}{2}},\varphi)+
  \gamma_{2}(\bar{\theta}_{v}^{n+\frac{1}{2}},\varphi)
  +(R^{n+\frac{1}{2}}_{1},\varphi) =0 ,  \label{29}
\end{align}
where $
R_{1}^{n+\frac{1}{2}}= \bar{R}^{n+\frac{1}{2}}\bigg(\frac{F'(\tilde{u}_{N}^{n+\frac{1}{2}})}{\sqrt{E(\tilde{u}_{N}^{n+\frac{1}{2}})}}- \frac{F'(\tilde{U}^{n+\frac{1}{2}})}{\sqrt{E(\tilde{U}^{n+\frac{1}{2}})}}\bigg) $.

{\color{blue}Following from \eqref{244}}, we have $\theta_{u}^{0}=0$ and $\theta_{v}^{0}=0$. Moreover, for the first step, we have
\begin{eqnarray}
 &&(\frac{\tilde{\theta}_{u}^{\frac{1}{2}}}{\tau/2},\psi)=(\tilde{\theta}_{v}^{\frac{1}{2}},\psi), \label{31}\\
  &&(\frac{\tilde{\theta}_{v}^{\frac{1}{2}}}{\tau/2},\varphi)+\kappa ((-\Delta)^{\frac{\a}{2}}\tilde{\theta}_{u}^{\frac{1}{2}},\varphi)+\gamma_{1}((-\Delta)^{\frac{\a}{2}}\tilde{\theta}_{v}^{\frac{1}{2}},\varphi)
  +\gamma_{2}(\tilde{\theta}_{v}^{\frac{1}{2}},\varphi)
  +(\tilde{R}_{1}^{\frac{1}{2}},\varphi) =0 ,  \label{32}
\end{eqnarray}
where $\tilde{R}_{1}^{\frac{1}{2}}=\tilde{R}^{\frac{1}{2}}\bigg(\frac{F'(u_{N}^{0})}{\sqrt{E(u_{N}^{0})}}-\frac{F'(U^{0})}{\sqrt{E(U^{0})}}\bigg)$.
\begin{Theo} \label{theo32}
Suppose that  \eqref{sinene1}-\eqref{sinene4} have unique solutions satisfying \eqref{bound1} and $u_{N}^{n},~v_{N}^{n}$ are the solutions of \eqref{22}-\eqref{24}, $0\leq n \leq K$. Then there exist two positive constants $\tau^{*}$ and $N^{*}$, such that when $\tau \leq \tau^{*}$ and $N\geq N^{*}$, we have
{\color{blue}
\begin{eqnarray}
&& \|\theta_{u}^{n}\|^2+\|\theta_{v}^{n}\|^2+\kappa |\theta_{u}^{n}|_{\frac{\a}{2}}^2 \leq CN^{-2\a}, 0\leq n \leq K,  \label{33}\\
&&  \|u_{N}^{n}\|_{\infty} \leq M_{2}. \label{34}
\end{eqnarray}
}
\end{Theo}
\begin{proof}
  1. The first step: estimate $\tilde{\theta}_{u}^{\frac{1}{2}}$ and $\tilde{\theta}_{v}^{\frac{1}{2}}$.

   Letting $\psi=\tilde{\theta}_{u}^{\frac{1}{2}}$ and $\varphi=\tilde{\theta}_{v}^{\frac{1}{2}}$ in \eqref{31}-\eqref{32}, we can get
\begin{align}
&\|\tilde{\theta}_{u}^{\frac{1}{2}}\|^2=\frac{\tau}{2}(\tilde{\theta}_{u}^{\frac{1}{2}},\tilde{\theta}_{v}^{\frac{1}{2}}),  \label{35} \\
& \|\tilde{\theta}_{v}^{\frac{1}{2}}\|^2+\frac{\tau}{2}\kappa ((-\Delta)^{\frac{\a}{2}}\tilde{\theta}_{u}^{\frac{1}{2}},\tilde{\theta}_{v}^{\frac{1}{2}})
+\frac{\tau}{2}\gamma_{1}|\tilde{\theta}_{v}^{\frac{1}{2}}|_{\frac{\a}{2}}^2
  +\frac{\tau}{2}\gamma_{2}\|\tilde{\theta}_{v}^{\frac{1}{2}}\|^2
 =-\frac{\tau}{2}(\tilde{R}_{1}^{\frac{1}{2}},\tilde{\theta}_{v}^{\frac{1}{2}}).\label{36}
\end{align}
Setting $\psi=(-\Delta)^{\frac{\a}{2}}\tilde{\theta}_{u}^{\frac{1}{2}}$ in  \eqref{35}, the \eqref{36} can be rewritten as
\begin{equation}
   \|\tilde{\theta}_{v}^{\frac{1}{2}}\|^2+\kappa |\tilde{\theta}_{u}^{\frac{1}{2}}|_{\frac{\a}{2}}^2
+\frac{\tau}{2}\gamma_{1}|\tilde{\theta}_{v}^{\frac{1}{2}}|_{\frac{\a}{2}}^2
  +\frac{\tau}{2}\gamma_{2}\|\tilde{\theta}_{v}^{\frac{1}{2}}\|^2
 =-\frac{\tau}{2}(\tilde{R}_{1}^{\frac{1}{2}},\tilde{\theta}_{v}^{\frac{1}{2}}), \label{37}
\end{equation}
meanwhile $\tilde{R}_{1}^{\frac{1}{2}}$ also could be  represented by
  \begin{align}
\tilde{R}_{1}^{\frac{1}{2}}=\tilde{R}^{\frac{1}{2}}\bigg(
\frac{F'(u_{N}^{0})-F'(U^{0})}{\sqrt{E(u_{N}^{0})}}+
\frac{F'(U^{0})(E(u_{N}^{0})
-E(U^{0}))}
{\sqrt{E(u_{N}^{0})}\sqrt{E(U^{0})}
\bigg(\sqrt{E(u_{N}^{0})}+\sqrt{E(U^{0})}\bigg)}\bigg). \label{38}
  \end{align}
  From Theorem \ref{theo31} and   
  $(\tilde{e}_{r}^{\frac{1}{2}})^2\leq C\tau^{3}$, thus  $|\tilde{R}^{\frac{1}{2}}|$ is bounded. We denote
  \begin{eqnarray}
  \hat{K}= \max\{|\tilde{r}^{\frac{1}{2}}|, |\tilde{R}^{\frac{1}{2}}| \}+1, \label{39}
  \end{eqnarray}
Combining \eqref{35}-\eqref{39}, we can get
\begin{equation}\label{40}
\|\tilde{\theta}_{u}^{\frac{1}{2}}\|^2+  \|\tilde{\theta}_{v}^{\frac{1}{2}}\|^2+\kappa |\tilde{\theta}_{u}^{\frac{1}{2}}|_{\frac{\a}{2}}^2\leq C\tau^2\|\tilde{\theta}_{v}^{\frac{1}{2}}\|^2+C\|\rho_{u}^{0}\|^{2}+C\tau\|\tilde{\theta}_{u}^{\frac{1}{2}}\|^2+C\tau\|\tilde{\theta}_{v}^{\frac{1}{2}}\|^2.
\end{equation}
Thus, by Lemma \ref{le1.5}, we can have
\begin{equation}\label{41}
  \|\tilde{\theta}_{u}^{\frac{1}{2}}\|^2+  \|\tilde{\theta}_{v}^{\frac{1}{2}}\|^2+\kappa |\tilde{\theta}_{u}^{\frac{1}{2}}|_{\frac{\a}{2}}^2\leq N^{-2\a}.
\end{equation}
From Lemma \ref{le1.4}, we have
\begin{equation}\label{42}
  \|\tilde{u}_{N}^{\frac{1}{2}}\|_{\infty}\leq \|P_{N}\tilde{U}^{\frac{1}{2}}\|_{\infty}+\|\tilde{\theta}_{u}^{\frac{1}{2}}\|_{\infty}\leq \|P_{N}\tilde{U}^{\frac{1}{2}}\|_{\infty}+CN^{1-\a}\leq M_{2},
\end{equation}
where $\tau \leq \tau^{*}$ and $N \geq C^{\frac{1}{\a-1}}$.

2. The second step: estimate $\theta_{u}^{1}$ and $\theta_{v}^{1}$. 

Letting $\psi=\theta_{u}^{1}$ and $\varphi=\theta_{v}^{1}$ in \eqref{28}-\eqref{29}, then it leads to
\begin{align}
&\|\theta_{u}^{1}\|^2= \frac{\tau}{2}(\theta_{v}^{1},\theta_{u}^{1}), \label{43}\\
  &\|\theta_{v}^{1}\|^2+\kappa |\theta_{u}^{1}|_{\frac{\a}{2}}^2+\frac{\tau}{2}\gamma_{1}|\theta_{v}^{1}|_{\frac{\a}{2}}^2+\gamma_{2}\frac{\tau}{2}\|\theta_{v}^{1}\|^2
  =\frac{\tau}{2}(R^{\frac{1}{2}}_{1},\theta_{v}^{1}) .  \label{44}
%
\end{align}
Using
\begin{equation*}
   R_{1}^{\frac{1}{2}}=\bar{R}^{\frac{1}{2}}\bigg(\frac{F'(\tilde{u}_{N}^{\frac{1}{2}})-F'(\tilde{U}^{\frac{1}{2}})}{\sqrt{E(\tilde{u}_{N}^{\frac{1}{2}})}}+
\frac{F'(\tilde{U}^{\frac{1}{2}})(E(\tilde{u}_{N}^{\frac{1}{2}})
-E(\tilde{U}^{\frac{1}{2}}))}
{\sqrt{E(\tilde{U}^{\frac{1}{2}})}\sqrt{E(\tilde{u}_{N}^{\frac{1}{2}})}
\bigg(\sqrt{E(\tilde{u}_{N}^{\frac{1}{2}})}+\sqrt{E(\tilde{U}^{\frac{1}{2}})}\bigg)}\bigg),\label{45}
\end{equation*}
and from Theorem \ref{theo31}, we  
get the boundedness of $|\bar{R}^{\frac{1}{2}}|$. We denote
  $\hat{K}_{1}= 
  \max\{|\bar{R}^{\frac{1}{2}}|,  |\bar{r}^{\frac{1}{2}}| \}+1$, 
  and by virtue of Lemma \ref{le1.5}, we have
\begin{equation}\label{46}
  \|\theta_{u}^{1}\|^2+\|\theta_{v}^{1}\|^2+\kappa |\theta_{u}^{1}|_{\frac{\a}{2}}^2\leq CN^{-2\a}.
\end{equation}
Moreover, when $\tau \leq \tau^{*}$ and $N \geq C^{\frac{1}{\a-1}}$,
\begin{equation}\label{47}
  \|u_{N}^{1}\|_{\infty}\leq \|P_{N}U^{1}\|_{\infty}+\|\theta_{u}^{1}\|_{\infty}\leq \|P_{N}U^{1}\|_{\infty}+CN^{1-\a}\leq M_{2}.
\end{equation}

3. The third step. Suppose that \eqref{33} holds for $0 \leq n \leq k$ and discuss that the case for $n=k+1$ holds.

 We use the mathematical induction method and assume
\begin{equation}\label{47}
  \|u_{N}^{n}\|_{\infty}\leq M_{2},~0 \leq n  \leq k,
\end{equation}
and denote
$\hat{K}_{2}= \max\{|\tilde{r}^{n+\frac{1}{2}}|, |\tilde{R}^{n+\frac{1}{2}}| (n=0,...,k) \}+1$.
Next, we intend to prove that \eqref{33} holds for $n=k+1$.
Letting $\psi=\tau \bar{\theta}_{u}^{k+\frac{1}{2}}$ and $\varphi=\tau \bar{\theta}_{v}^{k+\frac{1}{2}}$ in \eqref{28}-\eqref{29}, respectively, we  get
\begin{eqnarray}
&& \|\theta_{u}^{k+1}\|^2-\|\theta_{u}^{k}\|^2=\tau(\bar{\theta}_{v}^{k+\frac{1}{2}},\bar{\theta}_{u}^{k+\frac{1}{2}})
, \label{48}\\
  &&\|\theta_{v}^{k+1}\|^2-\|\theta_{v}^{k}\|^2+\kappa (|\theta_{u}^{k+1}|_{\frac{\a}{2}}^{2}-|\theta_{u}^{k}|_{\frac{\a}{2}}^{2})
  +\tau \gamma_{1}|\bar{\theta}_{v}^{k+\frac{1}{2}}|_{\frac{\a}{2}}^{2}
  +\gamma_{2}\tau \|\bar{\theta}_{v}^{k+\frac{1}{2}}\|^2
   =-\tau(R_{1}^{k+\frac{1}{2}},\bar{\theta}_{v}^{k+\frac{1}{2}}). \label{49}
\end{eqnarray}
By using Cauchy-Schwarz inequality and Young's inequality, we have
\begin{align}
 &|(R_{1}^{k+\frac{1}{2}},\bar{\theta}_{v}^{k+\frac{1}{2}})|\leq \|\theta_{u}^{k}\|^2+  \|\theta_{u}^{k-1}\|^2+\|\theta_{v}^{k}\|^2+\|\theta_{v}^{k+1}\|^2+\|\rho_{u}^{k}\|^2+\|\rho_{u}^{k-1}\|^2, \label{50} \\
& |(\bar{\theta}_{v}^{k+\frac{1}{2}},\bar{\theta}_{u}^{k+\frac{1}{2}})| \leq \frac{1}{4}\bigg(\|\theta_{u}^{k}\|^2+  \|\theta_{u}^{k+1}\|^2+\|\theta_{v}^{k}\|^2+\|\theta_{v}^{k+1}\|^2\bigg). \label{51}
\end{align}
Therefore, we get
\begin{eqnarray}
&&\|\theta_{u}^{k+1}\|^2-\|\theta_{u}^{k}\|^2+\|\theta_{v}^{k+1}\|^2-\|\theta_{v}^{k}\|^2+\kappa (|\theta_{u}^{k+1}|_{\frac{\a}{2}}^{2}-|\theta_{u}^{k}|_{\frac{\a}{2}}^{2})\leq \nonumber\\
&&C\tau\bigg(\|\theta_{u}^{k}\|^2+  \|\theta_{u}^{k+1}\|^2+\|\theta_{v}^{k}\|^2+\|\theta_{v}^{k+1}\|^2+\|\rho_{u}^{k}\|^2+\|\rho_{u}^{k-1}\|^2+ \|\theta_{u}^{k-1}\|^2\bigg). \label{52}
\end{eqnarray}
Summing from  $n=1$ to $n=k$ and using the discrete Gronwall's inequality and Lemma \ref{le1.5}, \eqref{52} leads to
\begin{equation}\label{53}
  \|\theta_{u}^{k+1}\|^2+\|\theta_{v}^{k+1}\|^2+\kappa |{\theta}_{u}^{k+1}|_{\frac{\a}{2}}^{2} \leq CN^{-2\a},
\end{equation}
which implies that we have completed the mathematics induction. Therefore \eqref{33} holds and then \eqref{34} is obtained similar as \eqref{47}. We have completed the proof.
 \end{proof}
 \section{The proof of the Theorem \ref{theo3.3}}
 For convenience, we denote
\begin{eqnarray*}
  &&e_{Nu}^{n}=u^{n}-u_{N}^{n}=u^{n}-P_{N}u^{n}+P_{N}u^{n}-u_{N}^{n}=\eta_{u}^{n}+\xi_{u}^{n}, \nonumber\\
 &&e_{Nv}^{n}=v^{n}-v_{N}^{n}=v^{n}-P_{N}v^{n}+P_{N}v^{n}-v_{N}^{n}=\eta_{v}^{n}+\xi_{v}^{n}.
\end{eqnarray*}
{\color{blue}{In the convergence analysis, we use the mathematical induction method.
From \eqref{sin1}-\eqref{sin2}, \eqref{sinetay11}-\eqref{sinetay33} and \eqref{25}-\eqref{27},  the error equations at $t=t_{\frac{1}{2}}$ are the following:
\begin{eqnarray}
 &&(\frac{\tilde{\xi}_{u}^{\frac{1}{2}}}{\tau/2},\psi)=(\tilde{\xi}_{v}^{\frac{1}{2}},\psi)+(\tilde{Q}_{1}^{\frac{1}{2}},\psi), \label{add31}\\
  &&(\frac{\tilde{\xi}_{v}^{\frac{1}{2}}}{\tau/2},\varphi)+\kappa ((-\Delta)^{\frac{\a}{2}}\tilde{\xi}_{u}^{\frac{1}{2}},\varphi)+\gamma_{1}((-\Delta)^{\frac{\a}{2}}\tilde{\xi}_{v}^{\frac{1}{2}},\varphi)
  +\gamma_{2}(\tilde{\xi}_{v}^{\frac{1}{2}},\varphi)
  +(\tilde{R}_{2}^{\frac{1}{2}},\varphi) =(\tilde{Q}_{2}^{\frac{1}{2}},\varphi) ,  \label{add32}\\
  && \frac{\tilde{e}_{r}^{\frac{1}{2}}}{\tau/2}=\tilde{F}_{1}^{\frac{1}{2}} +\tilde{Q}_{3}^{\frac{1}{2}}, \label{add33}
\end{eqnarray}
where
\begin{small}
\begin{align}
&\tilde{R}_{2}^{\frac{1}{2}}= \tilde{r}^{\frac{1}{2}}\frac{F'(u^{0})}{\sqrt{E(u^{0})}}- \tilde{R}^{\frac{1}{2}}\frac{F'(u_{N}^{0})}{\sqrt{E(u_{N}^{0})}}=
\tilde{e}_{r}^\frac{1}{2}\frac{F'(u^{0})}{\sqrt{E(u^{0})}}+\tilde{R}^{\frac{1}{2}}(\frac{F'(u^{0})}{\sqrt{E(u^{0})}}-\frac{F'(u_{N}^{0})}{\sqrt{E(u_{N}^{0})}}),   \label{add34} \\
&\tilde{F}_{1}^{\frac{1}{2}} = \frac{1}{2\sqrt{E(u^{0})}}\Big(F'(u^{0}),\frac{\tilde{u}^{\frac{1}{2}}-u^{0}}{\tau/2}\Big)
  -\frac{1}{2\sqrt{E(u_{N}^{0})}}\Big(F'(u_{N}^{0}),\frac{\tilde{u}_{N}^{\frac{1}{2}}-u_{N}^{0}}{\tau/2}\Big)
  \nonumber \\
 &=\Big(\frac{1}{2\sqrt{E(u^{0})}}F'(u^{0})-\frac{1}{2\sqrt{E(u_{N}^{0})}}F'(u_{N}^{0}),\frac{\tilde{u}^{\frac{1}{2}}-u^{0}}{\tau/2}\Big)
 +\frac{1}{2\sqrt{E(u_{N}^{0})}}\Big(F'(u_{N}^{0}),\frac{\tilde{e}_{u}^{\frac{1}{2}}-e_{u}^{0}}{\tau/2}\Big).\label{add35}
\end{align}
\end{small}
Letting $\psi=2\tau \tilde{\xi}_{u}^{\frac{1}{2}}$ and $\varphi=2\tau \tilde{\xi}_{v}^{\frac{1}{2}}$ in \eqref{add31}-\eqref{add32}, and multiplying $2\tau\tilde{e}_{r}^{\frac{1}{2}}$ on the both sides of \eqref{add33}, and using $$(\tilde{\xi}_{u}^{\frac{1}{2}},(-\Delta)^{\frac{\a}{2}}\tilde{\xi}_{u}^{\frac{1}{2}})
=2\tau(\tilde{\xi}_{v}^{\frac{1}{2}},(-\Delta)^{\frac{\a}{2}}\tilde{\xi}_{u}^{\frac{1}{2}})
+2\tau(\tilde{Q}_{1}^{\frac{1}{2}},(-\Delta)^{\frac{\a}{2}}\tilde{\xi}_{u}^{\frac{1}{2}}),$$  we have
\begin{eqnarray}
&& \|\tilde{\xi}_{u}^{\frac{1}{2}}\|^2=2\tau(\tilde{\xi}_{v}^{\frac{1}{2}},\tilde{\xi}_{u}^{\frac{1}{2}})
+2\tau(\tilde{Q}_{1}^{\frac{1}{2}},\tilde{\xi}_{u}^{\frac{1}{2}}), \label{add36}\\
  &&\|\tilde{\xi}_{v}^{\frac{1}{2}}\|^2+\kappa |\tilde{\xi}_{u}^{\frac{1}{2}}|_{\frac{\a}{2}}^2
  +2\tau \gamma_{1}|\tilde{\xi}_{v}^{\frac{1}{2}}|_{\frac{\a}{2}}^{2}+2\gamma_{2}\tau \|\tilde{\xi}_{v}^{\frac{1}{2}}\|^2 \nonumber \\
  &&
   =2\tau(\tilde{Q}_{2}^{\frac{1}{2}},\tilde{\xi}_{v}^{\frac{1}{2}})-2\tau (\tilde{R}_{2}^{\frac{1}{2}},\tilde{\xi}_{v}^{\frac{1}{2}}) +2\tau \kappa ((-\Delta)^{\frac{\a}{4}}\tilde{Q}_{1}^{\frac{1}{2}},(-\Delta)^{\frac{\a}{4}}\tilde{\xi}_{u}^{\frac{1}{2}}), \label{add37}\\
  &&(\tilde{e}_{r}^{\frac{1}{2}})^{2}=2\tau(\tilde{F}_{1}^{\frac{1}{2}},\tilde{e}_{r}^{\frac{1}{2}}) +2\tau(\tilde{Q}_{3}^{\frac{1}{2}},\tilde{e}_{r}^{\frac{1}{2}}).\label{add38}
\end{eqnarray}
From \eqref{add34}-\eqref{add38}, $|\tilde{\eta}_{u}^{\frac{1}{2}}/\frac{\tau}{2}|\leq N^{-m}\tau^{-1}/2\int_{t_{0}}^{t_{\frac{1}{2}}}\|u_{t}\|^2ds$, using the Cauchy-Schwarz inequality and Young's inequality, we could have
\begin{equation*}
 |\tilde{e}_{r}^{\frac{1}{2}}|^2+\|\tilde{\xi}_{u}^{\frac{1}{2}}\|^2+\|\tilde{\xi}_{v}^{\frac{1}{2}}\|^2+\kappa |\tilde{\xi}_{u}^{\frac{1}{2}}|_{\frac{\a}{2}}^2 \leq C\tau K_{*},
\end{equation*}
where $K_{*}=\|\eta_{u}^{0}\|^{2}+N^{-2m}\tau^{-1}/2\int_{t_{0}}^{t_{\frac{1}{2}}}\|u_{t}\|^2ds+\|\tilde{Q}_{3}^{\frac{1}{2}}\|^2
+\|\tilde{Q}_{1}^{\frac{1}{2}}\|^2+\|\tilde{Q}_{2}^{\frac{1}{2}}\|^2+|\tilde{Q}_{1}^{\frac{1}{2}}|_{\a/2}^2+\tau^{3}\int_{t_{0}}^{t_{\frac{1}{2}}}\|\frac{\partial^{3}u}{\partial t^{3}}(s)\|^2ds $. Then we get $|\tilde{e}_{r}^{\frac{1}{2}}|^2+\|\tilde{\xi}_{u}^{\frac{1}{2}}\|^2+\|\tilde{\xi}_{v}^{\frac{1}{2}}\|^2+\kappa |\tilde{\xi}_{u}^{\frac{1}{2}}|_{\frac{\a}{2}}^2 \leq C(\tau^{3}+N^{-2m})$.

Then, we prove the equation \eqref{55} holds for $n=1$. From \eqref{sin1}-\eqref{sin2}, \eqref{sinetay1}-\eqref{sinetay3} and \eqref{22}-\eqref{24}, the error equations for $n=0$ follow that
\begin{align}
&(\delta_{t}\xi_{u}^{\frac{1}{2}},\psi)=(\bar{\xi}_{v}^{\frac{1}{2}},\psi)+(Q_{1}^{\frac{1}{2}},\psi), \label{add39}\\
  &(\delta_{t}\xi_{v}^{\frac{1}{2}},\varphi)+\kappa ((-\Delta)^{\frac{\a}{2}}\bar{\xi}_{u}^{\frac{1}{2}},\varphi)+\gamma_{1}((-\Delta)^{\frac{\a}{2}}\bar{\xi}_{v}^{\frac{1}{2}},\varphi)+
  \gamma_{2}(\bar{\xi}_{v}^{\frac{1}{2}},\varphi) 
 +(R^{\frac{1}{2}}_{2},\varphi) =(Q_{2}^{\frac{1}{2}},\varphi) ,  \label{add40}\\
 &\delta_{t}e_{r}^{\frac{1}{2}}=\bar{F}^{\frac{1}{2}} +Q_{3}^{\frac{1}{2}}, \label{add41}
\end{align}
where
\begin{small}
\begin{align}
&R_{2}^{\frac{1}{2}}= \bar{r}^{\frac{1}{2}}\frac{F'(\tilde{u}^{\frac{1}{2}})}{\sqrt{E(\tilde{u}^{\frac{1}{2}})}} - \bar{R}^{\frac{1}{2}}\frac{F'(\tilde{u}_{N}^{n+\frac{1}{2}})}{\sqrt{E(\tilde{u}_{N}^{\frac{1}{2}})}},   \label{59} \\
&\bar{F}^{\frac{1}{2}}=\frac{1}{2\sqrt{E(\tilde{u}^{n+\frac{1}{2}})}}(F'(\tilde{u}^{\frac{1}{2}}),\delta_{t}u^{\frac{1}{2}})-
\frac{1}{2\sqrt{E(\tilde{u}_{N}^{\frac{1}{2}})}}(F'(\tilde{u}_{N}^{\frac{1}{2}}),\delta_{t}u_{N}^{\frac{1}{2}}).\label{60}
\end{align}
\end{small}
Since the proof for $n=0$ is similar to the following proof, we just presents the results for k=0. It follows that
\begin{eqnarray*}
  &&  |e_{r}^{1}|^2+  \|\xi_{u}^{1}\|^2+\|\xi_{v}^{1}\|^2+\kappa |\xi_{u}^{1}|_{\frac{\a}{2}}^{2}\leq
  C\tau(\|\eta_{v}^{0}\|^{2}+\|\eta_{u}^{0}\|^{2}
+\|Q_{2}^{\frac{1}{2}}\|^2+\|Q_{1}^{\frac{1}{2}}\|^2+\|Q_{3}^{\frac{1}{2}}\|^2+|Q_{1}^{\frac{1}{2}}|_{\a/2}^2+\|\tilde{\eta}_{u}^{\frac{1}{2}}\|^{2}
+\|\tilde{\xi}_{u}^{\frac{1}{2}}\|^{2} \nonumber\\
 && +\tau^{3}\int_{t_{0}}^{t_{1}}\|\frac{\partial^{3}u}{\partial t^{3}}(s)\|^2ds+CN^{-2m} \tau^{-1}\int_{t_{0}}^{t_{1}}\|u_{t}\|^2ds)
 \leq  C(\tau^4+N^{-2m}).
\end{eqnarray*}

Next, we assume the \eqref{55}  holds for $n=2,\ldots, k$ and will prove that the equation \eqref{55} holds for $n=k+1$.}}
From \eqref{sin1}-\eqref{sin2}, \eqref{sinetay1}-\eqref{sinetay3} and \eqref{22}-\eqref{24}, we can get
\begin{align}
&(\delta_{t}\xi_{u}^{n+\frac{1}{2}},\psi)=(\bar{\xi}_{v}^{n+\frac{1}{2}},\psi)+(Q_{1}^{n+\frac{1}{2}},\psi), \label{56}\\
  &(\delta_{t}\xi_{v}^{n+\frac{1}{2}},\varphi)+\kappa ((-\Delta)^{\frac{\a}{2}}\bar{\xi}_{u}^{n+\frac{1}{2}},\varphi)+\gamma_{1}((-\Delta)^{\frac{\a}{2}}\bar{\xi}_{v}^{n+\frac{1}{2}},\varphi)+
  \gamma_{2}(\bar{\xi}_{v}^{n+\frac{1}{2}},\varphi) 
 +(R^{n+\frac{1}{2}}_{2},\varphi) =(Q_{2}^{n+\frac{1}{2}},\varphi) ,  \label{57}\\
 &\delta_{t}e_{r}^{n+\frac{1}{2}}=\bar{F}^{n+\frac{1}{2}} +Q_{3}^{n+\frac{1}{2}}, \label{58}
\end{align}
where
\begin{small}
\begin{align}
&R_{2}^{n+\frac{1}{2}}= \bar{r}^{n+\frac{1}{2}}\frac{F'(\tilde{u}^{n+\frac{1}{2}})}{\sqrt{E(\tilde{u}^{n+\frac{1}{2}})}} - \bar{R}^{n+\frac{1}{2}}\frac{F'(\tilde{u}_{N}^{n+\frac{1}{2}})}{\sqrt{E(\tilde{u}_{N}^{n+\frac{1}{2}})}},   \label{59} \\
&\bar{F}^{n+\frac{1}{2}}=\frac{1}{2\sqrt{E(\tilde{u}^{n+\frac{1}{2}})}}(F'(\tilde{u}^{n+\frac{1}{2}}),\delta_{t}u^{n+\frac{1}{2}})-
\frac{1}{2\sqrt{E(\tilde{u}_{N}^{n+\frac{1}{2}})}}(F'(\tilde{u}_{N}^{n+\frac{1}{2}}),\delta_{t}u_{N}^{n+\frac{1}{2}}).\label{60}
\end{align}
\end{small}
{\color{blue}{
Letting $\psi=2\tau \bar{\xi}_{u}^{n+\frac{1}{2}}$ and $\varphi=2\tau \bar{\xi}_{v}^{n+\frac{1}{2}}$ in \eqref{56}-\eqref{57}, and multiplying $2\tau\bar{e}_{r}^{n+\frac{1}{2}}$ on the both sides of \eqref{58}, and using $$(\delta_{t}\xi_{u}^{n+\frac{1}{2}},(-\Delta)^{\frac{\a}{2}}\bar{\xi}_{u}^{n+\frac{1}{2}})
=(\xi_{v}^{n+\frac{1}{2}},(-\Delta)^{\frac{\a}{2}}\bar{\xi}_{u}^{n+\frac{1}{2}})
+(Q_{1}^{n+\frac{1}{2}},(-\Delta)^{\frac{\a}{2}}\bar{\xi}_{u}^{n+\frac{1}{2}}),$$  we can get
\begin{eqnarray}
&& \|\xi_{u}^{n+1}\|^2-\|\xi_{u}^{n}\|^2=2\tau(\bar{\xi}_{v}^{n+\frac{1}{2}},\bar{\xi}_{u}^{n+\frac{1}{2}})
+2\tau(Q_{1}^{n+\frac{1}{2}},\bar{\xi}_{u}^{n+\frac{1}{2}}), \label{61}\\
  &&\|\xi_{v}^{n+1}\|^2-\|\xi_{v}^{n}\|^2+\kappa (|\xi_{u}^{n+1}|_{\frac{\a}{2}}^{2}-|\xi_{u}^{n}|_{\frac{\a}{2}}^{2})
  +2\tau \gamma_{1}|\bar{\xi}_{v}^{n+\frac{1}{2}}|_{\frac{\a}{2}}^{2}+2\gamma_{2}\tau \|\bar{\xi}_{v}^{n+\frac{1}{2}}\|^2 \nonumber \\
  &&
   =2\tau(Q_{2}^{n+\frac{1}{2}},\bar{\xi}_{v}^{n+\frac{1}{2}})-2\tau (R_{2}^{n+\frac{1}{2}},\bar{\xi}_{v}^{n+\frac{1}{2}}) +2\tau \kappa ((-\Delta)^{\frac{\a}{4}}Q_{1}^{n+\frac{1}{2}},(-\Delta)^{\frac{\a}{4}}\bar{\xi}_{u}^{n+\frac{1}{2}}), \label{62}\\
  &&(e_{r}^{n+1})^{2}-(e_{r}^{n})^{2}=2\tau(\bar{F}^{n+\frac{1}{2}},\bar{e}_{r}^{n+\frac{1}{2}}) +2\tau(Q_{3}^{n+\frac{1}{2}},\bar{e}_{r}^{n+\frac{1}{2}}),\label{63}
\end{eqnarray}
where
\begin{small}
\begin{align}
 & R_{2}^{n+\frac{1}{2}}=r^{n+\frac{1}{2}}(\frac{F'(\tilde{u}^{n+\frac{1}{2}})}{\sqrt{E(\tilde{u}^{n+\frac{1}{2}})}}
  -\frac{F'(\tilde{u}_{N}^{n+\frac{1}{2}})}{\sqrt{E(\tilde{u}_{N}^{n+\frac{1}{2}})}})+
e_{r}^{n+\frac{1}{2}}\frac{F'(\tilde{u}_{N}^{n+\frac{1}{2}})}{\sqrt{E(\tilde{u}_{N}^{n+\frac{1}{2}})}}, \label{64}\\
&\bar{F}^{n+\frac{1}{2}}=\bigg(\frac{1}{2\sqrt{E(\tilde{u}^{n+\frac{1}{2}})}}F'(\tilde{u}^{n+\frac{1}{2}})
-\frac{1}{2\sqrt{E(\tilde{u}_{N}^{n+\frac{1}{2}})}}F'(\tilde{u}_{N}^{n+\frac{1}{2}}),\delta_{t}u^{n+\frac{1}{2}}\bigg)
+\bigg(\frac{1}{2\sqrt{E(\tilde{u}_{N}^{n+\frac{1}{2}})}}F'(\tilde{u}_{N}^{n+\frac{1}{2}}),\delta_{t}\bar{e}_{u}^{n+\frac{1}{2}}\bigg),\label{65}\\
&\frac{1}{\sqrt{E(\tilde{u}^{n+\frac{1}{2}})}}F'(\tilde{u}^{n+\frac{1}{2}})
-\frac{1}{\sqrt{E(\tilde{u}_{N}^{n+\frac{1}{2}})}}F'(\tilde{u}_{N}^{n+\frac{1}{2}})=
 \nonumber\\
&\frac{F'(\tilde{u}^{n+\frac{1}{2}})-F'(\tilde{u}_{N}^{n+\frac{1}{2}})}{\sqrt{E(\tilde{u}^{n+\frac{1}{2}})}}+
\frac{F'(\tilde{u}_{N}^{n+\frac{1}{2}})(E(\tilde{u}^{n+\frac{1}{2}})-E(\tilde{u}_{N}^{n+\frac{1}{2}}))}
{\sqrt{E(\tilde{u}_{N}^{n+\frac{1}{2}})}\sqrt{E(\tilde{u}^{n+\frac{1}{2}})}
\bigg(\sqrt{E(\tilde{u}^{n+\frac{1}{2}})}+\sqrt{E(\tilde{u}_{N}^{n+\frac{1}{2}})}\bigg)}, \label{66}\\
&\bigg(\frac{1}{2\sqrt{E(\tilde{u}_{N}^{n+\frac{1}{2}})}}F'(\tilde{u}_{N}^{n+\frac{1}{2}}),\delta_{t}\bar{e}_{u}^{n+\frac{1}{2}}\bigg)
=
\bigg(\frac{1}{2\sqrt{E(\tilde{u}_{N}^{n+\frac{1}{2}})}}F'(\tilde{u}_{N}^{n+\frac{1}{2}}),\delta_{t}\bar{\eta}_{u}^{n+\frac{1}{2}}+\delta_{t}\bar{\xi}_{u}^{n+\frac{1}{2}}\bigg)\nonumber\\
&=\bigg(\frac{1}{2\sqrt{E(\tilde{u}_{N}^{n+\frac{1}{2}})}}F'(\tilde{u}_{N}^{n+\frac{1}{2}}),\delta_{t}\bar{\eta}_{u}^{n+\frac{1}{2}}\bigg)
+\bigg(\frac{1}{2\sqrt{E(\tilde{u}_{N}^{n+\frac{1}{2}})}}F'(\tilde{u}_{N}^{n+\frac{1}{2}}),\bar{\xi}_{v}^{n+\frac{1}{2}}+Q_{1}^{n+\frac{1}{2}}\bigg). \label{add1}
\end{align}
\end{small}
From Theorem \ref{theo32}, we get $u_{N}^{n}$ and $\tilde{u}_{N}^{n+\frac{1}{2}}$ are bounded. Following the proof of \eqref{14} and from \eqref{65}-\eqref{add1}, \eqref{tay2}, we obtain
\begin{align}
&\bigg|\bigg(\frac{1}{2\sqrt{E(\tilde{u}_{N}^{n+\frac{1}{2}})}}F'(\tilde{u}_{N}^{n+\frac{1}{2}}),\delta_{t}\bar{\eta}_{u}^{n+\frac{1}{2}}\bigg)\bigg|\leq C(\|\tilde{u}_{N}^{n+\frac{1}{2}}\|_{\infty})|\delta_{t}\bar{\eta}_{u}^{n+\frac{1}{2}}|\leq  CN^{-m} \tau^{-1}\int_{t_{n}}^{t_{n+1}}\|u_{t}^{n+1/2}\|ds, \label{in12}\\
&\tau |(R_{2}^{n+\frac{1}{2}},\bar{\xi}_{v}^{n+\frac{1}{2}})|\leq C\tau\bigg( \|\xi_{v}^{n+1}\|^{2}+\|\xi_{v}^{n}\|^{2} +(e_{r}^{n+1})^{2}+(e_{r}^{n})^{2}+\|\tilde{\xi}_{u}^{n+\frac{1}{2}}\|^{2}\bigg)+\nonumber\\
&~~~~~~~~~~~~~~~~~~~~~~C\tau\bigg(\|\eta_{v}^{n+1}\|^{2}+\|\eta_{v}^{n}\|^{2}+\|\eta_{u}^{n}\|^{2}+\|\eta_{u}^{n-1}\|^{2}\bigg), \label{67}\\
&\tau|(\bar{F}^{n+\frac{1}{2}},\bar{e}_{r}^{n+\frac{1}{2}})|\leq C\tau\bigg(\|\tilde{\xi}_{u}^{n+\frac{1}{2}}\|^{2}+\|\tilde{\eta}_{u}^{n+\frac{1}{2}}\|^{2}+ \|\xi_{v}^{n+1}\|^{2}+\|\xi_{v}^{n}\|^{2}+|e_{r}^{n+1}|^2+|e_{r}^{n}|^2 \nonumber\\
&~~~~~+\|\eta_{v}^{n+1}\|^{2}+\|\eta_{v}^{n}\|^{2}+\|Q_{1}^{n+\frac{1}{2}}\|^2\bigg)+
C\tau^{4}\int_{t_{n}}^{t_{n+1}}\|\frac{\partial^{3}u}{\partial t^{3}}(s)\|^2ds +CN^{-2m}\int_{t_{n}}^{t_{n+1}}\|u_{t}^{n+1/2}\|^2ds. \label{68}
\end{align}
For convenience, we let {\color{blue}{$E^{n}:= |e_{r}^{n}|^2+ \|\xi_{u}^{n}\|^2 +\|\xi_{v}^{n}\|^2+\kappa |\xi_{u}^{n}|_{\frac{\a}{2}}^2$}}. Then, from \eqref{61}-\eqref{68}, we get
\begin{equation}
  E^{n+1}-E^{n}\leq C\tau (E^{n+1}+E^{n})+C\tau \mathfrak{R},
\end{equation}
where $\mathfrak{R}=\|\eta_{v}^{n+1}\|^{2}+\|\eta_{v}^{n}\|^{2}+\|\eta_{u}^{n}\|^{2}+\|\eta_{u}^{n-1}\|^{2}
+\|Q_{2}^{n+\frac{1}{2}}\|^2+\|Q_{1}^{n+\frac{1}{2}}\|^2+\|Q_{3}^{n+\frac{1}{2}}\|^2+|Q_{1}^{n+\frac{1}{2}}|_{\a/2}^2+\|\tilde{\eta}_{u}^{n+\frac{1}{2}}\|^{2}
+\|\tilde{\xi}_{u}^{n+\frac{1}{2}}\|^{2}
+\tau^{3}\int_{t_{n}}^{t_{n+1}}\|\frac{\partial^{3}u}{\partial t^{3}}(s)\|^2ds+CN^{-2m} \tau^{-1}\int_{t_{n}}^{t_{n+1}}\|u_{t}^{n+1/2}\|^2ds.$
By virtue of discrete Gronwall's inequality,
one arrives at $E^{k+1}\leq C(N^{-2m}+\tau^{4})$, which shows that the equation \eqref{55} holds  for $n=k+1$ and we have completed the mathematics induction.
Thus, without losing generality,  we conclude
\begin{equation}\label{69}
  |e_{r}^{n}|^2+  \|\xi_{u}^{n}\|^2+\|\xi_{v}^{n}\|^2+\kappa |\xi_{u}^{n}|^2_{\frac{\a}{2}}\leq C(\tau^4+N^{-2m}).
\end{equation}
Using Lemma \ref{le1.5} and the triangle inequality,  we have
\begin{equation}\label{add3}
  \|u^{n}-u_{N}^{n}\|^2+\|v^{n}-v_{N}^{n}\|^2+|r^{n}-R^{n}|^2\leq C(\tau^4+N^{-2m}), ~ \kappa|u^{n}-u_{N}^{n}|^2_{\frac{\a}{2}}\leq C(\tau^4+N^{\a-2m}).
\end{equation}
}}
{\color{blue}{ The proof of Theorem  \ref{theo3.3} is end when $\tau \leq \tau^{*}, ~N \geq N^{*}$.

For the other three cases, we denote $S_0:=(\tau^{*})^{4}+(N^{*})^{-2m} $.
\begin{itemize}
  \item[1.] When $\tau \geq \tau^{*}, ~N\leq N^{*}$,  there holds $\tau^{4}+N^{-2m} \geq S_{0}$. By using the continuous and discrete energy dissipation, i.e., $H(t)\leq H(t_{0})$ and $\mathbf{H}^{n}\leq \mathbf{H}^{0}$, we can get
\begin{align*}\label{70}
  & \|u^{n}-u_{N}^{n}\|^2+\|v^{n}-v_{N}^{n}\|^2+|r^{n}-R^{n}|^2\\
  & \leq \|u^{n}\|^2+\|u_{N}^{n}\|^2+\|v^{n}\|^2+\|v_{N}^{n}\|^2+|r^{n}|^2+|R^{n}|^2+\kappa|u^{n}|^2_{\frac{\a}{2}}+\kappa|u_{N}^{n}|^2_{\frac{\a}{2}}\\
  & \leq (\frac{2}{\kappa}+2)(H(t_{0})+\mathbf{H}^{0}) \leq C(\tau^{4}+N^{-2m}),\\
  & \kappa|u^{n}-u_{N}^{n}|_{\frac{\a}{2}}^2 \leq \kappa|u^{n}|^2_{\frac{\a}{2}}+\kappa|u_{N}^{n}|^2_{\frac{\a}{2}} \leq (\frac{2}{\kappa}+2)(H(t_{0})+\mathbf{H}^{0}) \leq C(\tau^{4}+N^{\a-2m}),
\end{align*}
where 
$C=(\frac{2}{\kappa}+2)(H(t_{0})+\mathbf{H}^{0})/S_{0}$.
  \item[2.] When $\tau \leq \tau^{*}, ~N \leq N^{*}$, it obviously follows that
$$C_{1}N^{-2m}\geq S_{0}\geq \tau^{4}+(N^{*})^{-2m},$$
where $C_{1}=\frac{(N^{*})^{-2m}+(\tau^{*})^{4}}{(N^{*})^{-2m}}$.
In this case, we also have
\begin{eqnarray*}
&& \|u^{n}-u_{N}^{n}\|^2+\|v^{n}-v_{N}^{n}\|^2+|r^{n}-R^{n}|^2
   \leq (\frac{2}{\kappa}+2)(H(t_{0})+\mathbf{H}^{0}) \leq C(\tau^{4}+N^{-2m}),\\
  && \kappa|u^{n}-u_{N}^{n}|_{\frac{\a}{2}} \leq C(\tau^{4}+N^{\a-2m}),
\end{eqnarray*}
where $C=(\frac{2}{\kappa}+2)(H(t_{0})+\mathbf{H}^{0})C_{1}/S_{0}$.
  \item[3.] When $\tau \geq \tau^{*}, ~N \geq  N^{*}$, there holds $$C_{2}\tau^{4}\geq S_{0}\geq (\tau^{*})^{4}+(N)^{-2m},$$
where $C_{2}=\frac{(N^{*})^{-2m}+(\tau^{*})^{4}}{(\tau^{*})^{4}}$.
In this case, we also have
\begin{eqnarray*}
&& \|u^{n}-u_{N}^{n}\|^2+\|v^{n}-v_{N}^{n}\|^2+|r^{n}-R^{n}|^2
   \leq (\frac{2}{\kappa}+2)(H(t_{0})+\mathbf{H}^{0}) \leq C(\tau^{4}+N^{-2m}),\\
  && \kappa|u^{n}-u_{N}^{n}|_{\frac{\a}{2}} \leq C(\tau^{4}+N^{\a-2m}),
\end{eqnarray*}
where $C=(\frac{2}{\kappa}+2)(H(t_{0})+\mathbf{H}^{0})C_{2}/S_{0}$.
\end{itemize}
In this work, $\kappa$ is suitable positive constant and does not effect the error estimate of our scheme.}}
Therefore, we complete the proof of Theorem  \ref{theo3.3}. $\Box$
\begin{Rem}\label{remak1}
Since the FGWE reduces to classical generalized wave equation when $\a=2$, the  unconditional energy dissipation and convergence analysis for the FGWE in this paper can naturally {\color{blue} be} applied to that of conventional generalized wave equation.
\end{Rem}

\section{Numerical experiments}
In this section, we will present some numerical examples to confirm  the discrete energy dissipation property  and the  accuracy of the full discrete SAV Fourier spectral schemes.
\begin{Ex}\label{ex1}
 In this example,  we take $\Omega=(-16,16)\times(-16,16)$, $T=1$, $\kappa=1$.
we consider $F(u)=1-\cos u$, 
 and take the  initial value $u(x,y,0)=\sin(\pi x/16)\cos(\pi y/16)$ and the corresponding $u_{t}(x,y,0)=0$.
\begin{itemize}
  \item case 1: $\gamma_1=0,~\gamma_2=0.$
  \item case 2: $\gamma_1=1,~\gamma_2=1.$
\end{itemize}
\end{Ex}

Since we have not the exact solution, and thus we choose sufficiently small time step $K=1000$ and $N=256$ to get `exact' solution. For simplicity,  Table 1 for case 1 just presents  the errors of
 $\|e_{u}^{n}\|_{\infty}$, $\|e_{v}^{n}\|_{\infty}$,  $\|e_{r}^{n}\|_{\infty}$  in time and show that the  fully discrete SAV scheme is second order accuracy in time.  Table 2 presents the spatial error of case 1 and we get spectral accuracy. Similar to the computational accuracy for case 1, we also
 present the temporal accuracy in Table 3 and spatial accuracy for case 2 in Table 4. Next, we will verify the conservation or dissipation property of the SAV Fourier spectral method. Fig. \ref{fig:5.4.1} shows the time evolution of the discrete energy $\mathbf{H}^{n}$  with different {\color{blue}{values}} $\gamma_1$ and $\gamma_2$ for  $\a=1.2$ and $\a=1.8$ and associate errors of $\mathbf{H}^{n}$ for the conservation case is presented in Fig. \ref{fig:5.4.2}. We  observe that the damping {\color{blue}{parameters}} $\gamma_1$ and $\gamma_2$ efficiently effect the dissipation property for SAV scheme for long time simulation.\\
$\mathbf{Example~ 2 \mathbf{.}}$
\begin{eqnarray*}
&&u_{tt}+(-\Delta)^{\frac{\a}{2}}u+\gamma_{1}(-\Delta)^{\frac{\a}{2}}u_{t}+\gamma_{2}u_{t}+F'(u)=0, ~(x,y,t)\in \Omega\times (0,T], \nonumber\\
&&  u(x,y,0)=\frac{1}{2}\arctan(\exp(-\sqrt{x^{2}+y^{2}})),~u_{t}(x,y,0)=0,  \label{exam1}
\end{eqnarray*}
where $\Omega=(-10,10)\times (-10,10)$.
 We present the profiles of numerical solution $u_{N}$ with different potential energy  $F(u)=u^{2}(\frac{1}{4}u^{2}-\frac{1}{2})$  for $\alpha=1.2$ and different {\color{blue}values} of $\gamma_{1}=\gamma_{2}$ at $T=8$, which shows that $\gamma_{1}$, $\gamma_{2}$ have  impacts on the profiles of wave.   In addition, we need to confirm the fully discrete  energy dissipation-preserving or conservation property of numerical solution.  Figs. \ref{fig:5.4.3} presents the discrete energy $\mathbf{H}^{n}$  with different values of $\gamma_{1}$ and $\gamma_{2}$ in long time simulation. They also show that the influence of coefficients  $\gamma_{1}$ and $\gamma_{2}$  of the damping term, i.e., when $\gamma_{1},\gamma_{2} \rightarrow 0 $, the energy discrete energy $\mathbf{H}^{n}$ decays more slowly.
 \begin{table}[!htb]
\caption{The $L^{\infty}$ errors of $u$, $v$ and $r$ in time   for case 1 of Example \ref{ex1} at time $T=1$.}
\label{tab9}
\centering
\begin{tabular}{ccccccc}
\hline
\multicolumn{1}{c}{ \multirow{2}*{} }& \multicolumn{6}{c}{$\a=1.2$, ~$N=256$} \\
\hline
  $\Delta t$&$\|e_{u}\|_{\infty}$ &rate  &$\|e_{v}\|_{\infty}$ &rate   &$\|e_{r}\|_{\infty}$ &rate     \\
  \hline
1/10 &5.4953e-05  &      &3.8436e-05 &       &7.0286e-05  &        \\
1/20 &1.3671e-05  &2.0228&9.4748e-06 &2.0262 &1.8739e-05  &1.9071 \\
1/40 &3.4055e-06  &2.0400&2.4215e-06 &2.0200 & 4.8090e-06  &1.9622 \\
\hline
\multicolumn{1}{c}{ \multirow{2}*{} }& \multicolumn{6}{c}{$\a=1.5$, ~$N=256$} \\
\hline
  $\Delta t$&$\|e_{u}\|_{\infty}$ &rate  &$\|e_{v}\|_{\infty}$ &rate   &$\|e_{r}\|_{\infty}$ &rate     \\
  \hline
1/10 &2.6766e-05  &      &4.0322e-05 &       &9.4683e-05  &        \\
1/20 &6.5081e-06  &2.0400&1.0764e-05 &1.9053 &2.5043e-05  & 1.9186 \\
1/40 &1.6014e-06  &2.0229& 2.7703e-06 &1.9581 &6.4132e-06  &1.8011 \\
\hline
\multicolumn{1}{c}{ \multirow{2}*{} }& \multicolumn{6}{c}{$\a=1.8$, ~$N=256$} \\
\hline
    $\Delta t$&$\|e_{u}\|_{\infty}$ &rate  &$\|e_{v}\|_{\infty}$ &rate   &$\|e_{r}\|_{\infty}$ &rate     \\
  \hline

1/10 &1.1679e-05 &      &4.4344e-05 &              &1.0619e-04  &        \\
1/20 &2.6560e-06 &2.0070 &1.1815e-05 &2.0202       &2.8057e-05  &1.9202        \\
1/40 &6.2907e-07 &2.0051 &3.0389e-06 &1.9681      &7.1860e-06   &1.9650\\
\hline
\multicolumn{1}{c}{ \multirow{2}*{} }& \multicolumn{6}{c}{$\a=2.0$, ~$N=256$} \\
\hline
  $\Delta t$&$\|e_{u}\|_{\infty}$ &rate  &$\|e_{v}\|_{\infty}$ &rate   &$\|e_{r}\|_{\infty}$ &rate     \\
  \hline
1/10 &6.0094e-06  &       &4.6442e-05 &       &1.1018e-04   &     \\
1/20 &2.3176e-06 &1.3745  &1.2362e-05 &1.9095 &2.9115e-05   &1.9200  \\
1/40 &6.8385e-07 &1.7608  &3.1787e-06 &1.9594 &7.4597e-06   &1.9645   \\
\hline
\end{tabular}
\end{table}
\begin{table}[!htb]
\caption{The errors  $|e_{u}|_{\frac{\a}{2}}$ and $\|e_{u}\|$ in space for case 1  of Example \ref{ex1} at time $T=1$.}
\label{tab10}
\centering
\begin{tabular}{cccccccc}
\hline
 \multicolumn{1}{c}{} &N&$\|e_{u}\|_{\frac{\a}{2}}$ &rate  &$\|e_{v}\|$ &rate     \\
  \hline
  \multirow{4}*{$\a=1.2$}
&4 &2.4856e-02   &      &4.7569e-03 &       \\
&8 &1.2884e-03   &4.2699&1.6427e-04&4.8558  \\
&16 &1.5766e-07  &12.9965&3.3921e-08 &12.2416  \\
&32 &3.6791e-10  &8.7432&3.6608e-11 & 9.8558  \\
\hline
\multirow{4}*{$\a=1.5$}
&4 &2.0858e-02  &        &4.9288e-03 &       \\
&8 &1.2541e-03  &4.0559  &1.7487e-04 &4.8169  \\
&16 &1.7644e-07 &12.7952 &3.6692e-08 &12.2185  \\
&32 &3.6622e-10 &8.9123  &7.8163e-11 &8.8748 \\
\hline
  \multirow{4}*{$\a=1.8$}
&4 &1.7418e-02  &      &5.0534e-03 &       \\
&8 &1.2128e-03  &3.8442&1.8282e-04 &4.7888  \\
&16 &1.9317e-07 &12.6162&3.8413e-08 &12.2165  \\
&32 &3.3982e-10 &9.1509&7.3433e-11 &9.0310 \\
\hline
\multirow{4}*{$\a=2.0$}&4 &1.5416e-02  &      &5.1169e-03 &       \\
&8 &1.1831e-03  &3.7038&1.8693e-04 &4.7747  \\
&16 &2.0333e-07 &12.5065&3.9072e-08 &12.2241  \\
&32 &2.8541e-10 &9.4766&7.8163e-10 &5.6435  \\
\hline
\end{tabular}
\end{table}

\begin{table}[!htb]
\caption{The $L^{\infty}$ errors of $u$, $v$ and $r$ in time   for case 2 of Example \ref{ex1} at time $T=1$.}
\label{tab9}
\centering
\begin{tabular}{ccccccc}
\hline
\multicolumn{1}{c}{ \multirow{2}*{} }& \multicolumn{6}{c}{$\a=1.2$, ~$N=256$} \\
\hline
  $\Delta t$&$\|e_{u}\|_{\infty}$ &rate  &$\|e_{v}\|_{\infty}$ &rate   &$\|e_{r}\|_{\infty}$ &rate     \\
  \hline
1/10 &1.1203e-04  &      &1.2852e-04 &       &1.6015e-05  &        \\
1/20 &2.7924e-05  &2.0043&3.2131e-05 &1.9999 &3.9645e-05  &2.0142 \\
1/40 &6.9647e-06  &2.0034&8.0258e-06 &2.0012 &9.8517e-06  &2.0087 \\
\hline
\multicolumn{1}{c}{ \multirow{2}*{} }& \multicolumn{6}{c}{$\a=1.5$, ~$N=256$} \\
\hline
  $\Delta t$&$\|e_{u}\|_{\infty}$ &rate  &$\|e_{v}\|_{\infty}$ &rate   &$\|e_{r}\|_{\infty}$ &rate     \\
  \hline
1/10 &8.3862e-05  &      &9.8031e-05 &       &1.2404e-04  &        \\
1/20 &2.0836e-05  &2.0088&2.4533e-05 &1.9985 &3.0467e-05  &2.0255 \\
1/40 &5.1883e-06  &2.0057&6.1308e-06 &2.0005 &7.5528e-06  &2.0121 \\
\hline
\multicolumn{1}{c}{ \multirow{2}*{} }& \multicolumn{6}{c}{$\a=1.8$, ~$N=256$} \\
\hline
    $\Delta t$&$\|e_{u}\|_{\infty}$ &rate  &$\|e_{v}\|_{\infty}$ &rate   &$\|e_{r}\|_{\infty}$ &rate     \\
  \hline

1/10 &6.4938e-05 &      &7.9646e-05 &              &1.0115e-04  &        \\
1/20 &1.6072e-05 &2.0144 &1.9955e-05 &1.9985       &2.4672e-05  &2.0355        \\
1/40 &3.9939e-06 &2.0087 &6.1308e-06 &2.0005      &6.0930e-06   &2.0177\\
\hline
\multicolumn{1}{c}{ \multirow{2}*{} }& \multicolumn{6}{c}{$\a=2.0$, ~$N=256$} \\
\hline
  $\Delta t$&$\|e_{u}\|_{\infty}$ &rate  &$\|e_{v}\|_{\infty}$ &rate   &$\|e_{r}\|_{\infty}$ &rate     \\
  \hline
1/10 &5.5967e-05  &       &7.1465e-05 &       &9.0572e-05   &     \\
1/20 &1.3813e-05 &2.0185  &1.7920e-05 &1.9956 &2.2004e-05   &2.0413  \\
1/40 &3.4273e-06 &2.0108  &4.4825e-06 &1.9992 &5.4268e-06   &2.0196   \\
\hline
\end{tabular}
\end{table}
\begin{table}[!htb]
\caption{The errors  $|e_{u}|_{\frac{\a}{2}}$ and $\|e_{u}\|$ in space for case 2  of Example \ref{ex1} at time $T=1$.}
\label{tab10}
\centering
\begin{tabular}{cccccccc}
\hline
 \multicolumn{1}{c}{} &N&$\|e_{u}\|_{\frac{\a}{2}}$ &rate  &$\|e_{v}\|$ &rate     \\
  \hline
  \multirow{4}*{$\a=1.2$}
&4 &1.6873e-02   &      &2.6620e-03 &       \\
&8 &7.8623e-04   &4.4236&7.5409e-05&5.1416  \\
&16 &7.0047e-08  &13.4543&1.3512e-08 &12.4462  \\
&32 &1.6426e-10  &8.7361&4.7132e-11 &8.1633  \\
\hline
\multirow{4}*{$\a=1.5$}
&4 &1.4409e-02  &        &2.8344e-03 &       \\
&8 &7.7021e-04  &4.2255  &8.2058e-05 &5.1102  \\
&16 &7.3050e-08 &13.3640 &1.3431e-08 &12.5768  \\
&32 &2.6752e-10 &8.0930  &1.2130e-10 &6.7908 \\
\hline
  \multirow{4}*{$\a=1.8$}
&4 &1.2195e-02  &      &2.9658e-03 &       \\
&8 &7.4937e-04  &4.0244&8.7406e-05 &5.0845  \\
&16 &7.4484e-08 &13.2964&1.2990e-08 &12.7161  \\
&32 &1.7250e-10 &8.7541&1.1949e-10 &6.7643 \\
\hline
\multirow{4}*{$\a=2.0$}
&4 &1.0871e-02  &      &3.0354e-03 &       \\
&8 &7.3373e-04  &3.8891&9.0363e-05 &5.0700  \\
&16 &7.4571e-08 &13.2643&1.2546e-08 &12.8142  \\
&32 &4.3006e-10 &7.4381&1.1855e-10 &6.7255  \\
\hline
\end{tabular}
\end{table}
\begin{figure}[ht]
\centering
\includegraphics[height=2.1in,width=3.0in]{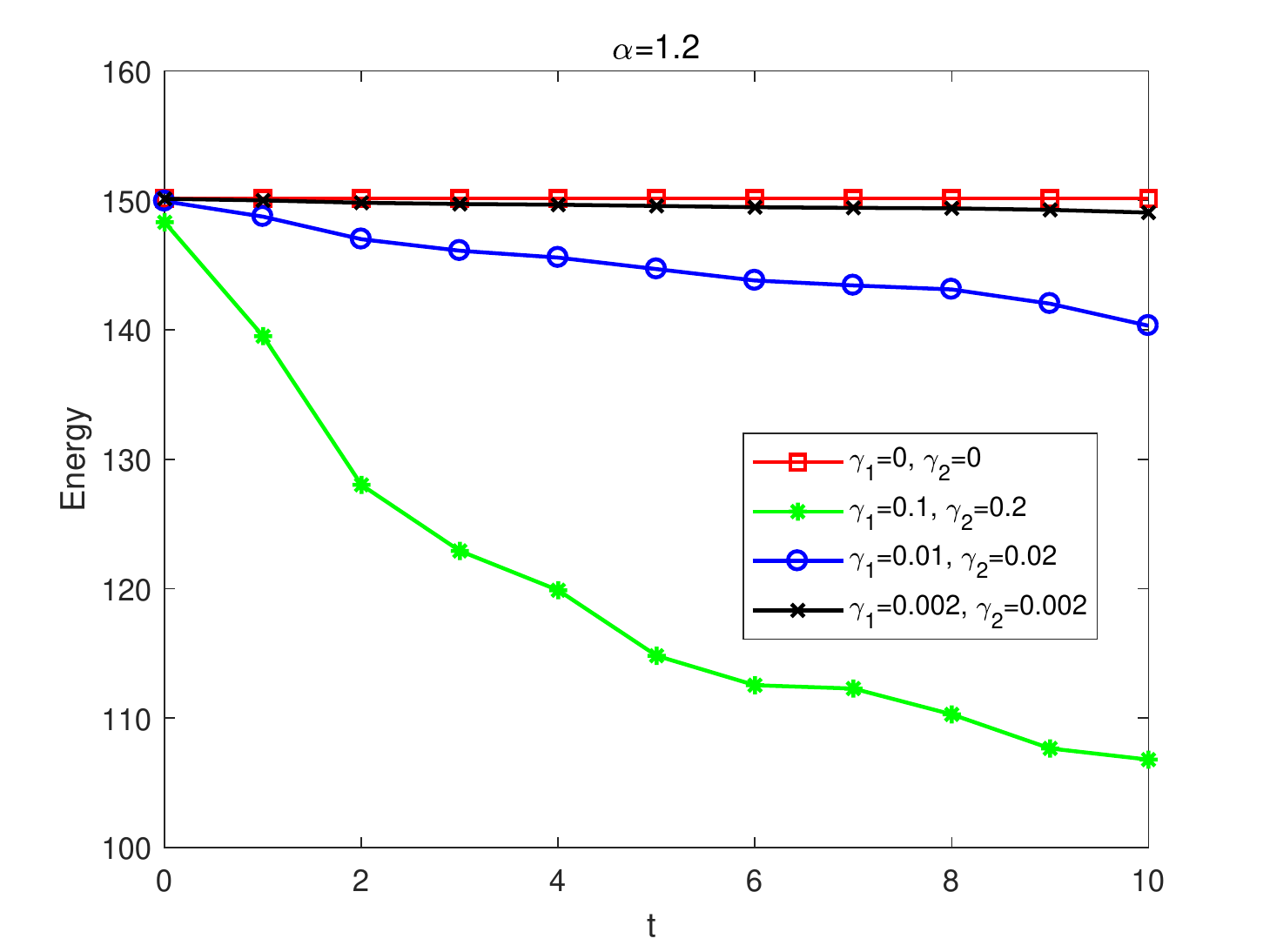}
\includegraphics[height=2.1in,width=3.0in]{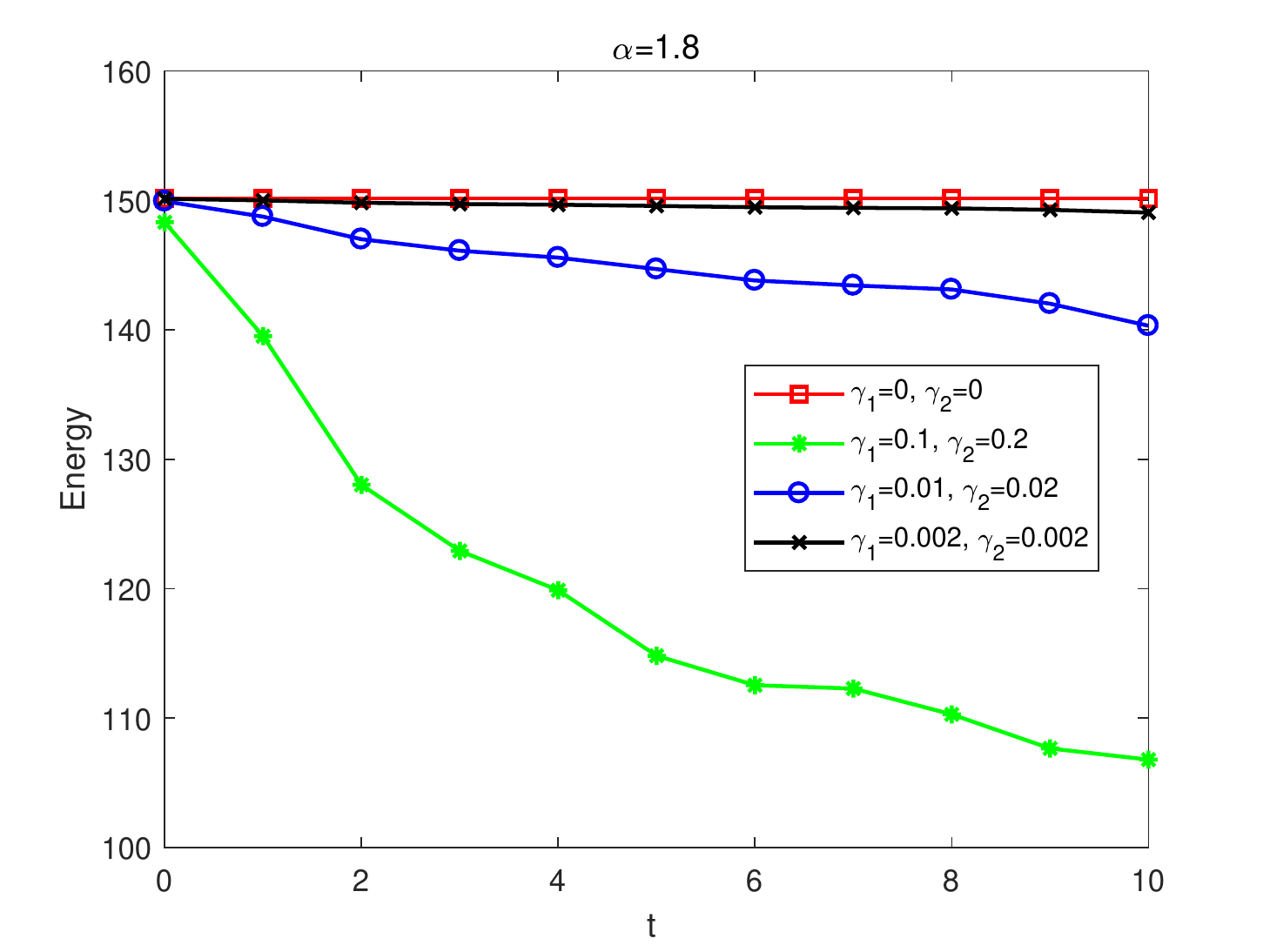}
 \caption{The values of the discrete energy $\mathbf{H}^{n}$ for  different $\gamma_{1}$ and $\gamma_{2}$ with time evolution for $\a_{1}=1.2$ and $\a_{2}=1.8$.}\label{fig:5.4.1}
\end{figure}
\begin{figure}[ht]
\centering
\includegraphics[height=2.1in,width=3.0in]{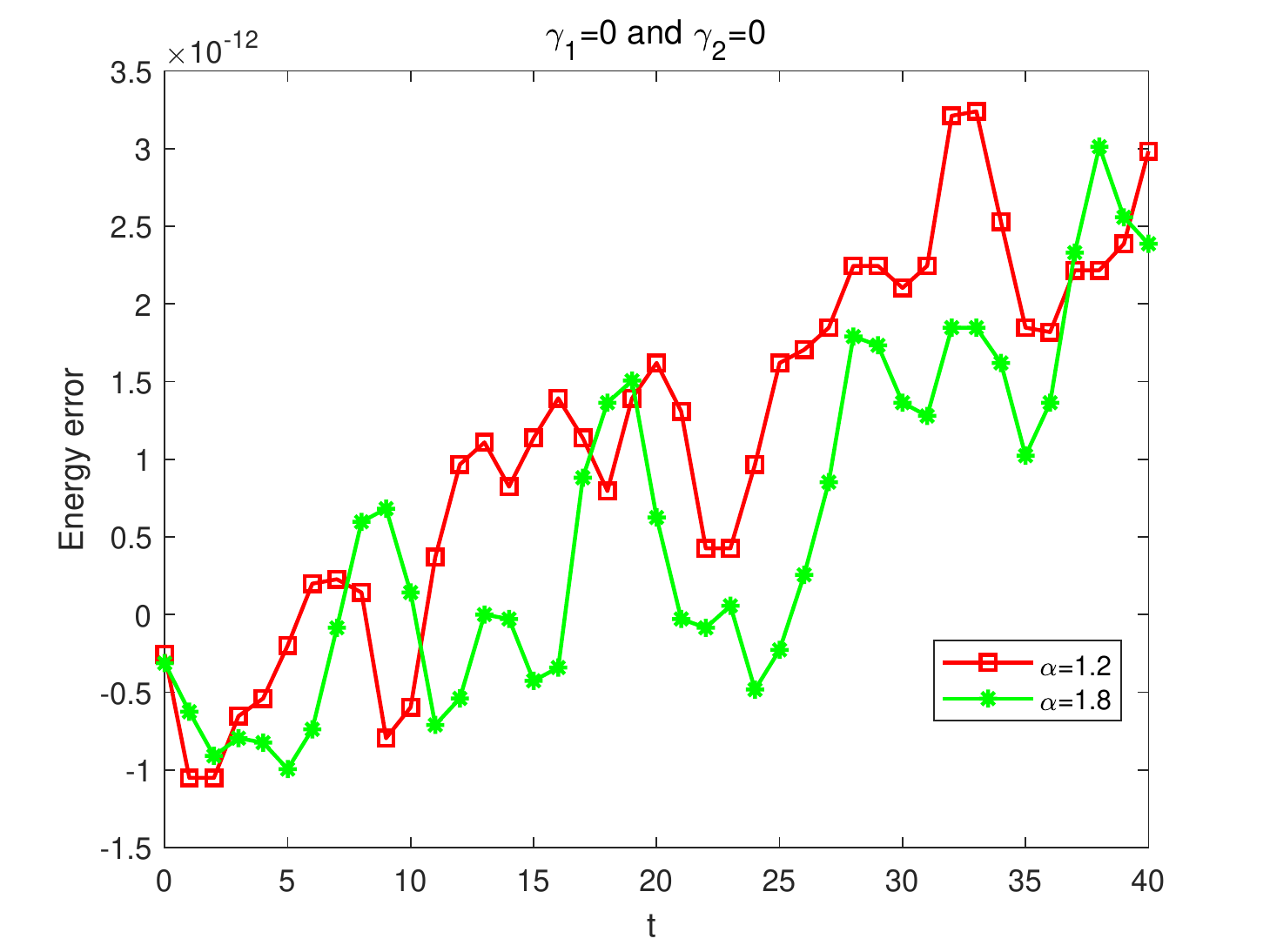}
 \caption{The error of the discrete energy $\mathbf{H}^{n}$ for  conservative form,i.e., $\gamma_{1}=0,~\gamma_{2}=0$ with time evolution for $\a_{1}=1.2$ and $\a_{2}=1.8$.}\label{fig:5.4.2}
\end{figure}
\begin{figure}[ht]
\centering
\includegraphics[height=2.1in,width=3.0in]{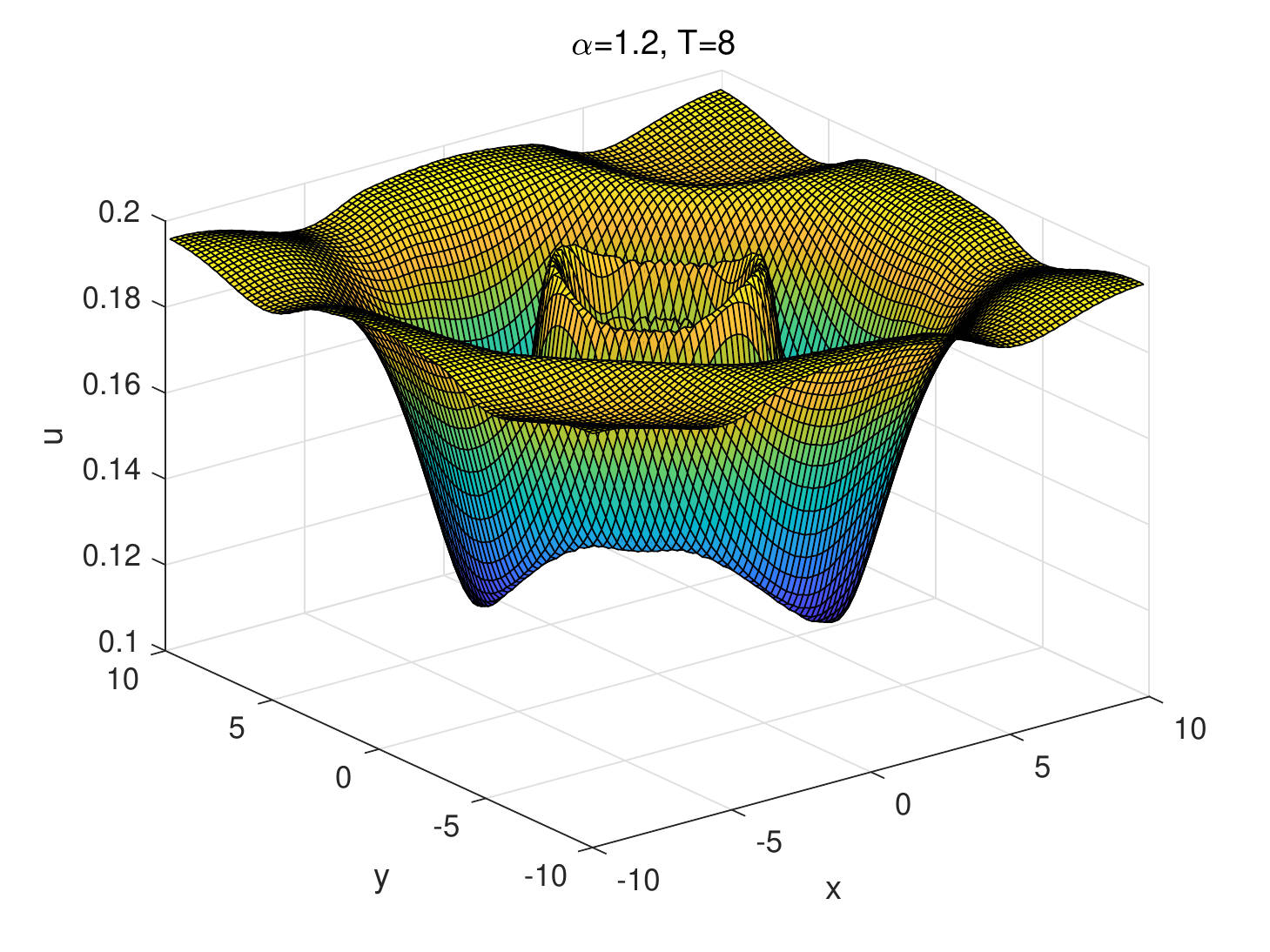}
\includegraphics[height=2.1in,width=3.0in]{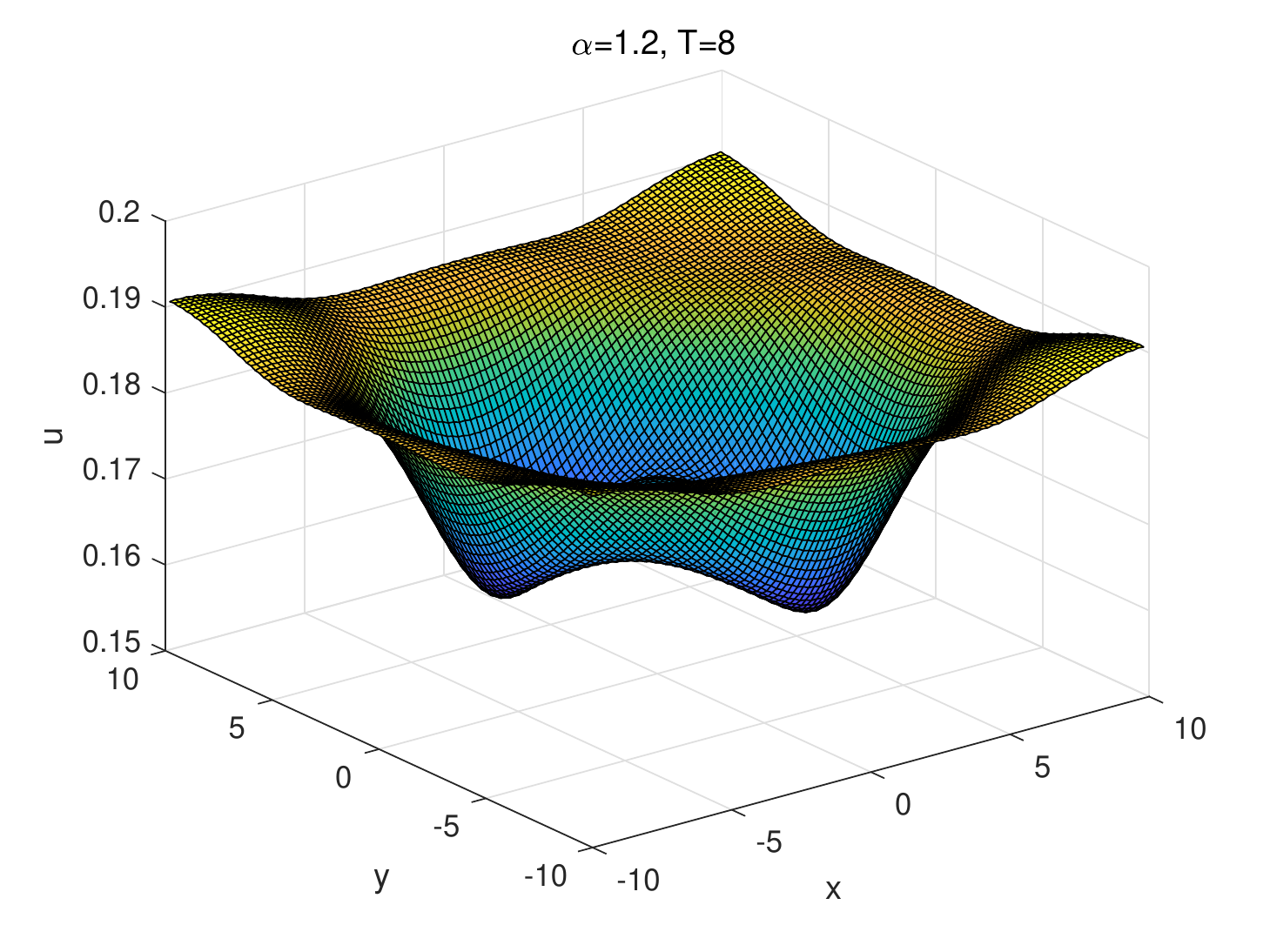}
 \caption{The numerical solution for  different $\gamma_{1}$ and $\gamma_{2}$ (Left: $\gamma_1=0$ and $\gamma_2=0$; Right: $\gamma_1=0.5$ and $\gamma_2=0$) with time evolution for $\a_{1}=1.2$ at $T=8$ .}\label{fig:5.4.4}
\end{figure}
\begin{figure}[ht]
\centering
\includegraphics[height=2.1in,width=3.0in]{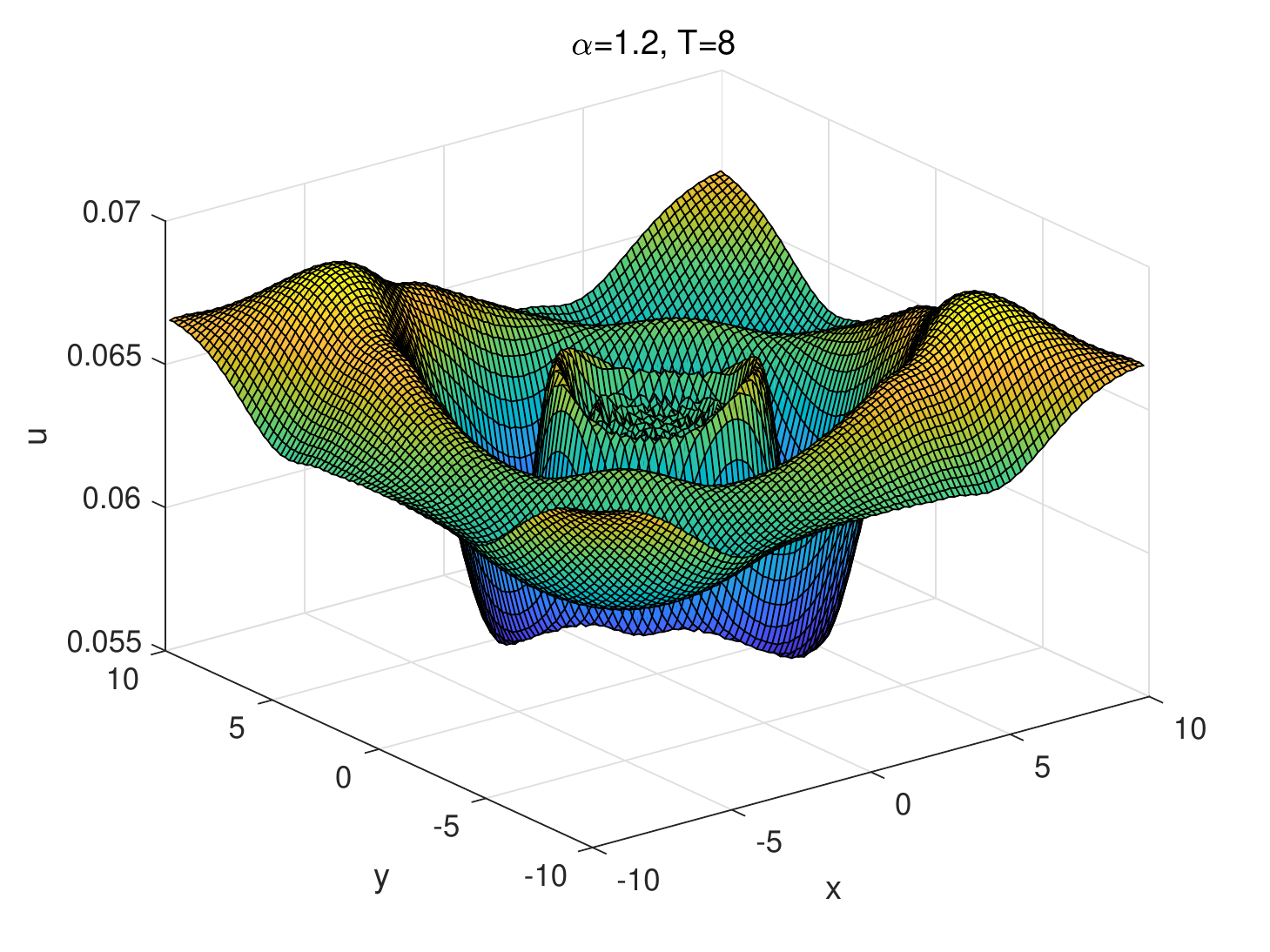}
\includegraphics[height=2.1in,width=3.0in]{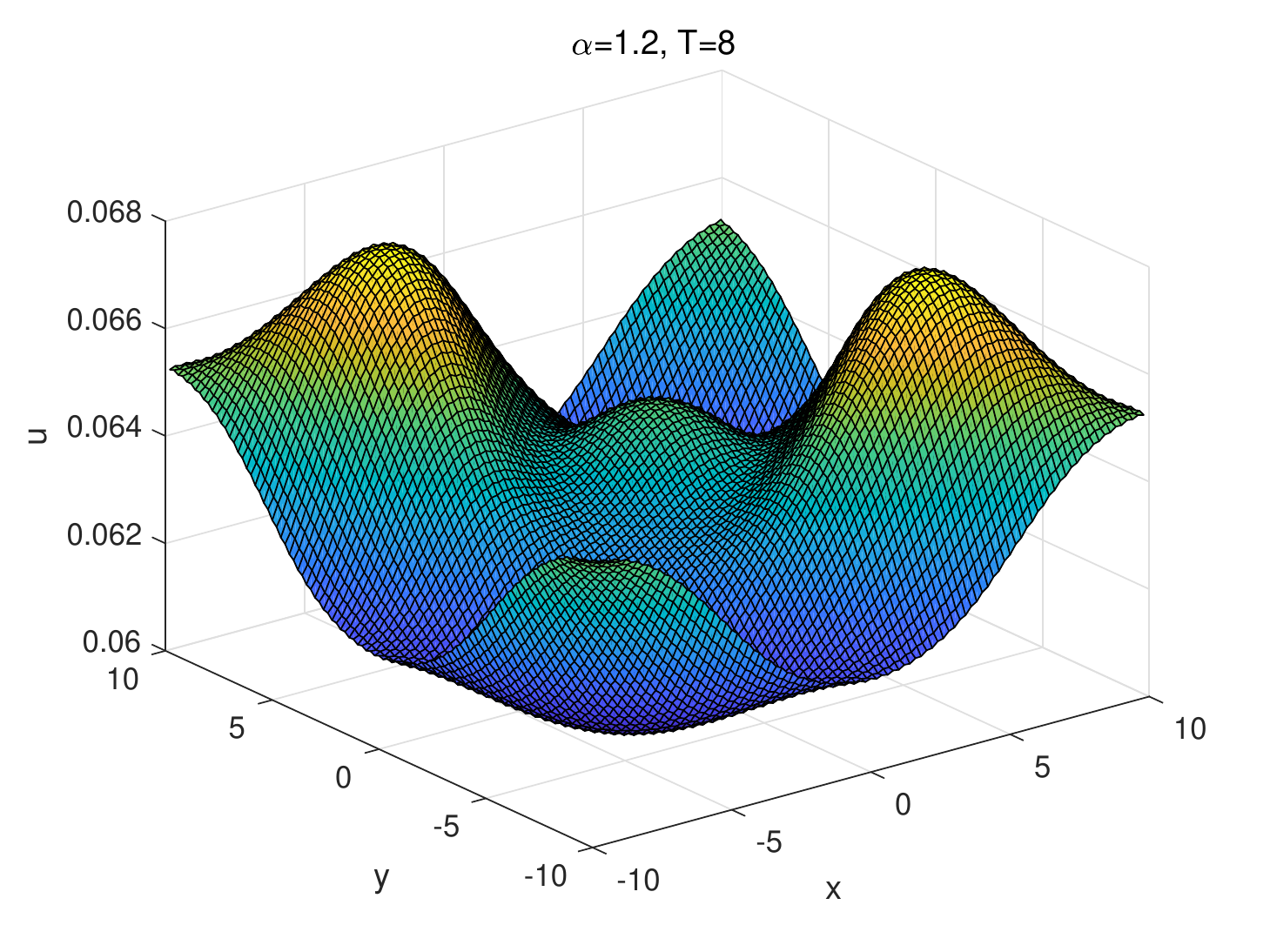}
 \caption{The numerical solution for  different $\gamma_{1}$ and $\gamma_{2}$  (Left: $\gamma_1=0$ and $\gamma_2=0.5$; Right: $\gamma_1=0.5$ and $\gamma_2=0.5$) with time evolution for $\a_{1}=1.2$ at $T=8$.}\label{fig:5.4.5}
\end{figure}
\begin{figure}[ht]
\centering
\includegraphics[height=2.1in,width=3.0in]{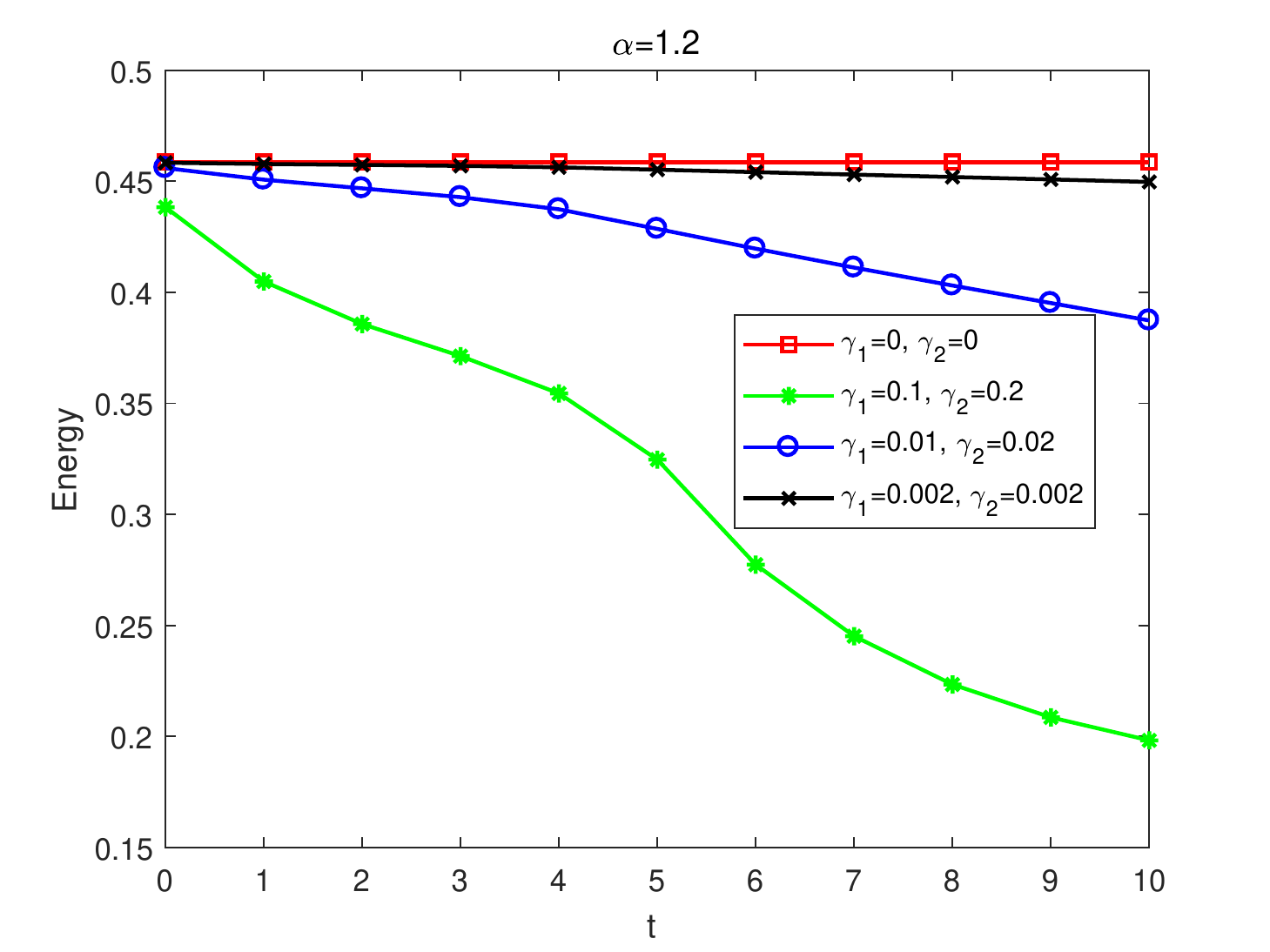}
\includegraphics[height=2.1in,width=3.0in]{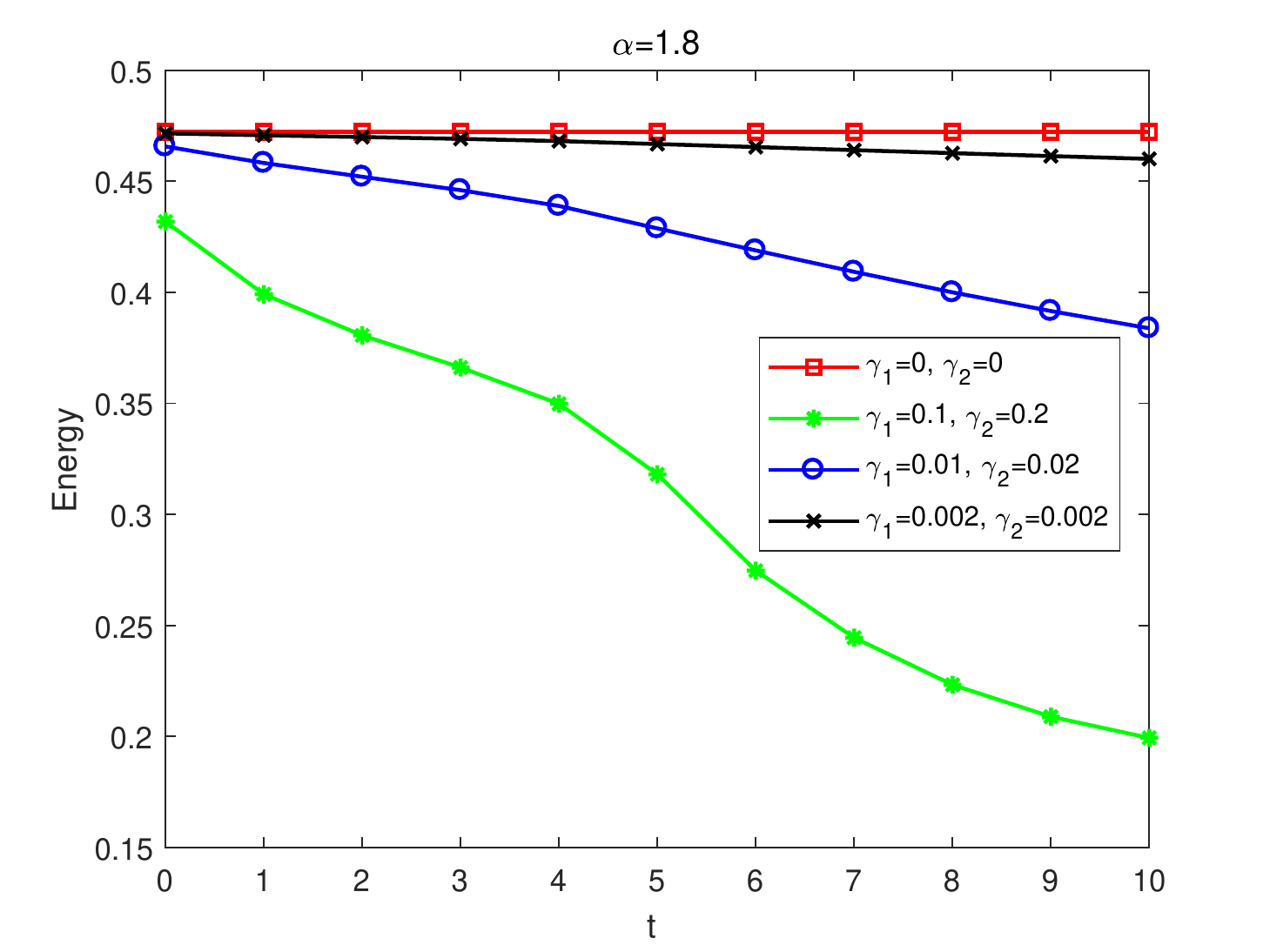}
 \caption{The values of the discrete energy $\mathbf{H}^{n}$ for  different $\gamma_{1}$ and $\gamma_{2}$ with time evolution for $\a_{1}=1.2$ and $\a_{2}=1.8$.}\label{fig:5.4.3}
\end{figure}

\section{Conclusion}
In this paper, we proposed {\color{blue}{a}} linearly implicit scheme to solve the nonlinear fractional general wave equation, and we used the SAV approach in time and Fourier spectral method in space. The resulting system is a linear system at each time step and  FFT solver can be adopted in space, so that the scheme could be efficient to implement.  The energy conservation or dissipation property of the scheme was strictly proved. In addition, we rigorously proved the unconditional convergence for the scheme by the temporal-spatial error splitting technique and got the second order accuracy in time and spectral accuracy in space. Finally, some numerical results were given to confirm our theoretical analysis.

\section*{Acknowledgment}
The authors thank the anonymous reviewer for excellent suggestions that helped improve this paper.

\section*{References}
\bibliographystyle{model1-num-names}
\bibliography{savsgorn}
\end{document}